\documentclass[12pt,reqno]{amsart}
\usepackage[foot]{amsaddr}
\usepackage{amsfonts,amssymb,amsmath}
\newtheorem{Definition}{Definition}[section]
\newtheorem{Theorem}{Theorem}[section]
\newtheorem{Lemma}{Lemma}[section]
\newtheorem{Corollary}{Corollary}[section]
\newtheorem{Proposition}{Proposition}[section]

\newtheorem{Remark}{Remark}[section]

\newcommand{\supp}{\operatorname{supp}}
\newcommand{\WF}{\operatorname{WF}}
\newcommand{\sing}{\operatorname{singsupp}}

\makeatletter
\@namedef{subjclassname@2020}{\textup{2020} Mathematics Subject Classification}
\makeatother

\begin{document}

\sloppy

\title[An introduction to extended Gevrey regularity]
{An introduction to extended Gevrey regularity}

\author{Nenad TEOFANOV \and Filip TOMI\'C \and Milica \v ZIGI\'C}

\address[Nenad Teofanov, Milica \v Zigi\'c]{Department of Mathematics and Informatics, Faculty of Sciences, University of Novi Sad, Serbia.}
\address[Filip Tomi\'c]{Faculty of Technical Sciences, University of Novi Sad, Serbia.}

\email[Corresponding author]{nenad.teofanov@dmi.uns.ac.rs}
\email{filip.tomic@uns.ac.rs}
\email{milica.zigic@dmi.uns.ac.rs}

\thanks{Corresponding author: Nenad Teofanov, nenad.teofanov@dmi.uns.ac.rs}

\date{April 2024.}

\begin{abstract}
Gevrey classes are the most common choice when considering the regularities of smooth functions that are not analytic. However, in various situations, it is important to consider smoothness properties that go beyond Gevrey regularity, for example when initial value problems are ill-posed in Gevrey settings. Extended Gevrey classes provide a convenient framework for studying smooth functions that possess weaker regularity than any Gevrey function. Since the available literature on this topic is scattered, our aim is to provide an overview to extended Gevrey regularity, highlighting its most important features. Additionally, we consider related dual spaces of ultradistributions and review some results on micro-local analysis in the context of extended Gevrey regularity. We conclude the paper with a few selected applications that may motivate further study of the topic.
\end{abstract}



\maketitle

%%%%%%%%%
\section{Introduction}\label{sec0}
%%%%%%%%%

Gevrey type regularity was introduced in the study of fundamental
solutions of the heat equation in \cite{Gevrey} and subsequently used to describe regularities stronger than smoothness ($C^{\infty}$-regula\-ri\-ty) and weaker than analyticity. This property turns out to be important in the general theory of linear partial differential equations, such as hypoellipticity, local solvability, and propagation of singularities, cf.  \cite{Rodino}.
In particular, the Cauchy problem for weakly hyperbolic linear partial differential equations (PDEs) can be well-posed in certain Gevrey classes, while at the same time being ill-posed in the class of analytic functions, as shown in \cite{ChenRodino, Rodino}.

Since there is a gap between Gevrey regularity and smoothness, it is important to study classes of smooth functions that do not belong to any Gevrey class. For example,  J\'ez\'equel  \cite{J} proved that the trace formula for Anosov flows in dynamical systems holds for certain intermediate regularity classes, and Cicognani and  Lorenz
used a different intermediate regularity when studying the well-posedness of strictly hyperbolic equations in \cite{CL}.

A systematic study of smoothness that goes beyond any Gevrey regularity was proposed in \cite{PTT-01, PTT-02}. This was accomplished by introducing two-parameter dependent sequences of the form $(p^{\tau p^{\sigma}})_p $, where $\tau >0$, $\sigma > 1 $. These sequences give rise to classes of ultradifferentiable functions $ {\mathcal E}_{\tau, \sigma}( \mathbb{R}^d)$, which differ from classical Carleman classes $C^L ( \mathbb{R}^d)$ (cf. \cite{H}), are larger than J\'ez\'equel's classes, and which go beyond Komatsu's approach to ultradifferentiable functions as described in, for example, \cite{Komatsuultra1}.
On one hand, these classes, called Pilipovi\'c-Teofanov-Tomi\'c classes in   \cite{Garrido-2}, serve as a prominent example of the generalized matrix approach to ulradifferentiable functions. On the other hand, they provide asymptotic estimates in terms of the Lambert functions, which have proven to be useful in various contexts, as discussed in \cite{CL,Garrido-1,TTT-24}.

Different aspects of the so-called {\em extended Gevrey regularity}, i.e., the regularity of ultradifferentiable functions from $ {\mathcal E}_{\tau, \sigma}( \mathbb{R}^d)$, have been studied in a dozen papers published in the last decade. Our aim is to offer a self-contained introduction to the subject and illuminate its main features. We provide proofs that, in general, simplify and complement those in the existing literature. Additionally, we present some new results, such as Proposition \ref{prop:extendedGervey-nizovi}, Proposition \ref{propalgebra}, and Theorem \ref{teoremaInverseClosedness} for the Beurling case, as well as Theorem \ref{Thm:Peli-Viner}.

This survey begins with preliminary Section \ref{sec1}, which covers the main properties of defining sequences, the Lambert function, and the associated function to a given sequence. We emphasize the remarkable connection between the associated function and the Lambert $W$ function (see Theorem \ref{TeoremaAsocirana}), which provides an elegant formulation of decay properties of the (short-time) Fourier transform of $f \in {\mathcal E}_{\tau, \sigma}(\mathbb{R}^d)$, as demonstrated in Proposition \ref{NASPW} and Corollary \ref{posledica1}.
In Section \ref{SectionSpaces}, we introduce the extended Gevrey classes ${\mathcal E}_{\tau, \sigma}( \mathbb{R}^d)$ and the corresponding spaces of ultradistributions. We then present their main properties, such as inverse closedness (Theorem \ref{teoremaInverseClosedness})
and the Paley-Wiener type theorem (Theorem \ref{Thm:Peli-Viner}).

In Section \ref{sec:WFs}, we give an application of extended Gevrey regularity in micro-local analysis. More precisely, we introduce wave-front sets, which detect singularities that are "stronger" than classical $C^\infty$ singularities and, at the same time, "weaker" than any Gevrey type singularity.

To provide a flavor of possible applications of extended Gevrey regularity, in Section \ref{sec:appl}, we briefly outline some results from \cite{CL} and \cite{Garrido-2}. More precisely, we present a result from \cite{CL} concerning the well-posedness of
strictly hyperbolic equations in ${\mathcal E}_{1, 2}( \mathbb{R}^d)$,
and observations from \cite{Garrido-2}, where the extended Gevrey classes are referred to as Pilipovi\'c-Teofanov-Tomi\'c classes and are considered within the extended matrix approach to ultradifferentiable classes.

We end this section by introducing some notation that will be used in the sequel.

%%%
\subsection{Notation} \label{notation}
%%%

We use the standard notation: $\mathbb{N}$, $\mathbb{N}_0$, $\mathbb{Z}$, $\mathbb{R}$,  $\mathbb{R}_+$, $\mathbb{C}$, denote sets of positive integers, non-negative integers, integers, real numbers, positive real numbers and complex numbers, respectively. The length of a multi-index $\alpha=(\alpha_1,\dots,\alpha_d) \in \mathbb{N}_0 ^d$ is denoted by $ |\alpha| = \alpha_1 + \alpha_2 + \dots + \alpha_d$ and $\alpha !:=\alpha_{1} ! \cdots \alpha_{d} !$. For $x=(x_1,\dots,x_d)\in\mathbb R^d$ we denote: $|x|:=\left(x_1^2+\ldots+x_d^2\right)^{1 / 2},$ $x^\alpha:=\prod_{j=1}^d x_j^{\alpha_j},$ and $D^\alpha=D_x^\alpha:=D_1^{\alpha_1} \cdots D_d^{\alpha_d}$, where $\displaystyle D_j^{\alpha_j}:=\left(-\frac{1}{2\pi i} \frac{\partial} { \partial x_j}\right)^{\alpha_j},$ $j=1, \dots, d$.

We write $ L^p (\mathbb{R}^d) $, $ 1\leq p\leq \infty$,  for the Lebesgue spaces, and  $ \mathcal{S} (\mathbb{R}^d)$ denotes the Schwartz space of infinitely smooth ($ C^\infty (\mathbb{R}^d)$) functions which, together with their derivatives, decay at infinity faster than any inverse polynomial. By $ \mathcal{S}' (\mathbb{R}^d)$ we denote the dual of $\mathcal{S} (\mathbb{R}^d)$, the space of tempered distributions, and $\mathcal D'(\mathbb R^d)$ is the dual of $\mathcal D(\mathbb R^d)=C^\infty _0 (\mathbb{R}^d)$, the space of compactly supported infinitely smooth functions.

We use brackets $\langle f,g\rangle$ to denote the extension of the inner product $\langle f,g\rangle=\int f(t){\overline {g(t)}}dt$ on $L^2(\mathbb{R}^d)$ to the dual pairing between a test function space $ \mathcal A $ and its dual $ {\mathcal A}' $: $ _{{\mathcal A}'}\langle \cdot, \overline{\cdot} \rangle _{\mathcal A}=(\cdot,\cdot).$

The notation $f = O (g) $ means that $ |f(x)| \leq C |g(x)| $ for some $C>0$ and $x$ in the intersection of domains for $f$ and $g$.
If $f=O(g)$ and $g=O(f)$, then we write $f\asymp g$.

The Fourier transform of $f \in L^1 (\mathbb{R}^d)$ given by
\begin{equation*}
\widehat{f} (\xi) := \int_{\mathbb{R}^d} f(x) e^{-2\pi i x \cdot \xi} dx, \quad \xi \in \mathbb{R}^d,
%\label{eqFourierTransform}
\end{equation*}
extends to $ L^2 (\mathbb{R}^d)$ by standard approximation procedure.

%{\col Proveriti da li FT iz prvih radova va\v zi kod Peli-Vinera}

The convolution between $f,g \in  L^1 (\mathbb{R}^d)$ is given by
$(f*g) (t) = \int f(x) g(t-x) dx$.

Translation, modulation, and dilation  operators, $T$, $M$, and $D$ respectively, when acting on $ f \in L^2 (\mathbb{R}^d)$ are defined by
\begin{equation*} %\label{trans-mod}
T_x f(\cdot) = f(\cdot - x) \quad \text{and} \quad
M_x f(\cdot) = e^{2\pi i x \cdot} f(\cdot), \quad
D_a  f(\cdot) = \frac{1}{a} f(\frac{\cdot}{a}),
\end{equation*}
$ x \in\mathbb{R}^d$, $ a >0$.
Then for $f,g \in L^2 (\mathbb{R}^d)$ the following relations hold:
\begin{equation*}
M_y T_x  = e^{2\pi i x \cdot y } T_x M_y, \quad
\widehat{(T_x f)} = M_{-x} \widehat f, \quad
\widehat{(M_x f)} = T_{x} \widehat f, \quad
x,y \in \mathbb{R}^d.
\end{equation*}

The Fourier transform, convolution, $T$, $M$, and $D$ are extended to other spaces of functions and distributions in a natural way.

%%%%%%%
\section{Preliminaries} \label{sec1}
%%%%%%%

%%%
\subsection{Defining sequences via Komatsu}
%%%

Komatsu's approach to the theory of ultradistributions (see \cite{Komatsuultra1}) is based on sequences of positive numbers  $(M_p) = (M_p)_{p\in \mathbb N_0}$, $ M_0 = 1$, which satisfy some of the following conditions: \newline
$(M.1) $ (logarithmic convexity)
\begin{equation*} \label{log-conv}
 M_p ^2 \leq M_{p-1}M_{p+1}, \qquad p\in \mathbb N;
\end{equation*}
$(M.2) $ (stability under the action of ultradifferentiable operators / convolution)
\begin{equation*} \label{ultradiff-stability}
(\exists A,B > 0) \quad M_{p+q}\leq AB^{p+q} M_p M_q,  \qquad p,q\in \mathbb N_0;
\end{equation*}
$(M.2)' $ (stability under the action of differentiable operators)
\begin{equation*} \label{ultradiff-stability2}
(\exists A,B > 0) \quad M_{p+1}\leq AB^{p} M_p,  \qquad p \in \mathbb N_0;
\end{equation*}
$(M.3)$ (strong  non-quasi-analyticity)
\begin{equation*} \label{strong-non-qa}
\sum\limits_{q=p+1} ^{\infty} \frac{M_{q-1}}{M_q}
\leq A\, p\, \frac{M_{p}}{M_{p+1}},
\qquad p \in \mathbb N;
\end{equation*}
$(M.3)'$  (non-quasi-analyticity)
\begin{equation*} \label{non-qa}
\sum\limits_{p=1}^{\infty}\frac{M_{p-1}}{M_p}<\infty.
\end{equation*}

Note that $(M.2) \Rightarrow (M.2)' $, and   $(M.3) \Rightarrow (M.3)'$.
In addition, $ (M.1) $ implies $ M_p M_q \leq   M_{p+q}$, $p,q\in \mathbb N_0$.

\par

Let $(M_p)$ be a positive monotone increasing sequence that satisfies $(M.1)$. Then
$ ( M_p /p! )^{1/p} ,$ $p\in\mathbb N$ is an almost increasing sequence if there exists $C>0$ such that
\begin{equation*}
\label{UsloviInveseClosed}
\left ( \frac{M_p}{p!}\right )^{1/p} \leq C \left ( \frac{M_q}{q!}\right )^{1/q}, \quad  p \leq q,\quad \text{and}\quad \lim_{p\to\infty} M_p^{1/p}=\infty.
\end{equation*}
This property is related to inverse closedness in  $C^{\infty}(\mathbb{R}^d)$,
see \cite{Siddiqi}.

The Gevrey sequence $M_p = p! ^s$, $p\in \mathbb N_0,$ $ s>1$ satisfies $  (M.1)$, $(M.2)$, and
$(M.3) $. It is also an almost increasing sequence.

\par

If $(M_p )$ and $(N_p )$ satisfy $(M.1)$, then we write $ M_p \subset N_p$ if there exist constants $A>0$ and $B>0$ (independent on $p$) such that
\begin{equation} \label{MpNp}
M_p \leq  A B^p N_p, \quad p\in \mathbb N.
\end{equation}
If, instead, for each $B>0$ there exists $A>0$ such that \eqref{MpNp} holds, then we write
\begin{equation*} \label{MpmanjeodNp}
M_p \prec N_p.
\end{equation*}

Assume that $(M_p )$ satisfies $ (M.1) $ and $ (M.3)' $. Then $ p! \prec M_p$.

\par

Let $\mathcal{R}$ denote the set of all sequences of positive numbers monotonically increasing to infinity. For a given sequence $(M_p )$ and $(r_p) \in \mathcal{R}$ we  consider
\begin{equation*} \label{Np}
N_0 = 1, \quad N_p = M_p r_1 r_2 \dots r_p = M_p \prod_{j=1} ^p r_j, \quad p\in  \mathbb N.
\end{equation*}
It is easy to see that if $(M_p )$ satisfies $ (M.1) $ and $ (M.3)' $, then $(N_p)$
satisfies $ (M.1) $ and $ (M.3)' $ as well. In addition, one can find  $(\tilde{r}_p) \in \mathcal{R}$
so that $ (M_p \prod_{j=1} ^p \tilde{r}_j)$ satisfies  $ (M.2) $ if $(M_p )$ does.
This follows from the next lemma.

\begin{Lemma} \label{comparison-sequences}
Let $(r_p) \in \mathcal{R}$ be given. Then there exists $(\tilde{r}_p) \in \mathcal{R}$ such that
$\tilde{r}_p \leq r_p$, $ p\in  \mathbb N $,
and
\begin{equation} \label{noviniz}
\prod_{j=1} ^{p+q} \tilde{r}_j \leq 2^{p+q}
\prod_{j=1} ^{p}  \tilde{r}_j \prod_{j=1} ^{q}  \tilde{r}_j, \qquad p,q\in \mathbb N.
\end{equation}
\end{Lemma}

\begin{proof}
It is enough to consider the sequence $(\tilde{r}_p)$ given by $\tilde{r}_1 = r_1$ and inductively
\begin{equation*}
\tilde{r}_{j+1} = \min \left\{ r_{j+1}, \frac{j+1}{j} \tilde{r}_j \right\}, \quad j \in \mathbb N.
\end{equation*}
Then $(\tilde{r}_p) \in \mathcal{R}$ and \eqref{noviniz} holds.
We refer to \cite[Lemma 2.3]{P-13} for details.
\end{proof}

\par

\subsection{Defining sequences for extended Gevrey regularity} \label{subsec:sequences}

To extend the class of Gevrey type ultradifferentiable functions we consider two-parameter sequences
of the form $M_p^{\tau,\sigma} =  p^{\tau p^{\sigma}}$, $p\in\mathbb N,$ $\tau>0$, $\sigma>1$.

>From Stirling's formula, and the fact that there exists $C>0$ (independent of $p$) such that
\begin{equation*}
s p \leq C\tau p^{\sigma},  \quad p \in \mathbb N,
\end{equation*}
for any $s, \sigma>1$, and $ \tau>0$, it follows that $ p!^s  \leq C_1 p^{\tau p^{\sigma}}$, for a suitable constant $C_1 > 0$.

The main properties of $(M_p^{\tau,\sigma})$ are collected in the next lemma (cf. \cite[Lemmas 2.2 and 3.1]{PTT-01}). The proof is given in the Appendix.

\begin{Lemma} \label{osobineM_p_s}
Let $\tau>0$, $\sigma>1$, $M_0 ^{\tau,\sigma}=1$,
and $M_p ^{\tau,\sigma}=p^{\tau p^{\sigma}}$, $p\in \mathbb N$.
Then the following properties hold:
\newline
$(M.1)$ \hspace{1em} $
(M_p^{\tau,\sigma})^2\leq M_{p-1}^{\tau,\sigma}M_{p+1}^{\tau,\sigma},  \quad
p\in \mathbb N, $ \newline
$\widetilde{(M.2)}$  \hspace{1em}   $
M_{p+q}^{\tau,\sigma}\leq C^{p^{\sigma} + q^{\sigma}}
M_p^{\tau 2^{\sigma-1},\sigma}M_q^{\tau 2^{\sigma-1},\sigma},
\quad p,q\in \mathbb N_0,\quad$
for some constant $ C\geq 1$,\newline
$\widetilde{(M.2)'}$ \hspace{1em}
$
M_{p+1}^{\tau,\sigma}\leq C ^{p^{\sigma}} M_p^{\tau,\sigma}, \quad p\in \mathbb N_0,
\quad$
for some constant $C\geq 1$, \newline
$(M.3)'$\hspace{1em}
$ \displaystyle   \sum\limits_{p=1}^{\infty}\frac{M_{p-1}^{\tau,\sigma}}{M_p^{\tau,\sigma}} <\infty.
$
%$$
%(M.3) \qquad
%\sum\limits_{q=p+1}^{\infty}\frac{M_{q-1}^{\tau,\sigma}}{M_q^{\tau,\sigma}} \leq C p
%\frac{M_{p}^{\tau 2^{\sigma -1},\sigma}}{M_{p+1} ^{\tau  2^{\sigma -1},\sigma}}.
%$$
\end{Lemma}

\begin{Remark} From the proof of $\widetilde{(M.2)'}$ it follows  that $(M_p^{\tau ,\sigma})$ does not satisfy $(M.2)'$, and therefore $(M.2)$ as well. One might expect that instead the sequence $(M_p^{\tau ,\sigma})$ satisfies
\begin{equation}
\label{naturalext2}
M_{p+q}^{\tau,\sigma}\leq C^{p^{\sigma}+q^{\sigma}}M_p^{\tau ,\sigma}M_q^{{\tau},\sigma},\quad p,q\in \mathbb N_0,
\end{equation}
for some constant $ C>0$. However, if we assume that \eqref{naturalext2} holds for e.g. $\tau=1$, then, for $p=q\not=0$, we obtain
\begin{equation*}
\label{formula} p^{(2p)^{\sigma}}\leq (C_1p)^{{2p}^{\sigma}},\quad p\in
\mathbb N, \quad \text{with} \quad C_1= C / 2^{2^{\sigma-1}}
\end{equation*}
which gives
\begin{equation*}
p^{2^{\sigma-1} p} \leq C_1,  \quad \text{ for all  } \quad p\in \mathbb N,
\end{equation*}
a contradiction.  Thus, $\widetilde{(M.2)}$ is a suitable alternative to $(M.2)$ when considering $(M_p^{\tau ,\sigma})$.
\end{Remark}

\par

Let $M_p^{\tau,\sigma} =  p^{\tau p^{\sigma}}$, $p\in\mathbb N,$ $\tau>0$, $\sigma>1$, and $(r_p) \in \mathcal{R}$. If $(\tilde{r}_p) \in \mathcal{R}$ is chosen as in Lemma \ref{comparison-sequences}, then the sequence $(N_p)$ given by
\begin{equation*}
N_0 = 1, \quad N_p =M_p^{\tau,\sigma}  \prod_{j=1} ^p \tilde{r}_j, \quad p\in  \mathbb N,
\end{equation*}
satisfies $ (M.1) $, $\widetilde{(M.2)}$, $\widetilde{(M.2)'}$, and $ (M.3)' $.

\par

We note that if $M_p^{\tau,\sigma}=p^{\tau p^{\sigma}}$, $p\in\mathbb N,$ $\tau>0$, $\sigma>1$,
then the sequence $\Big(\frac{M_p ^{\tau,\sigma}}{p^p}\Big)^{1/p},$ $p\in\mathbb N,$ is an almost increasing sequence since
$\Big(\frac{M_p ^{\tau,\sigma}}{p^p}\Big)^{1/p}=p^{\tau p^{\sigma-1}-1}$
and $ p^{\tau p^{\sigma-1}-1}<q^{\tau q^{\sigma-1}-1},$ $ \lceil (1/\tau)^{1/(\sigma-1)}\rceil < p <q.$

%%%
\subsection{The Lambert function}
\label{subsec:lambert}
%%%

The \emph{Lambert $W$ function} is defined as the inverse of $z e^{z}$, $z\in {\mathbb C}$. By $W(x)$, we denote the restriction of its principal branch to $[0,\infty)$. It is used as a convenient tool to describe asymptotic behavior in different contexts. We refer to \cite{LambF} for a review of some applications of the Lambert $W$
function in pure and applied mathematics, and to the recent monograph \cite{Mezo} for more details and generalizations.
It is noteworthy that the Lambert function describes the precise asymptotic behavior of associated function to the sequence $(M_p^{\tau,\sigma}) $. This fact was firstly observed in \cite{PTT-03}.

Some basic properties of the Lambert function $W$ are given below:
\newline
$(W1)$ \hspace{1em} $W(0)=0$, $W(e)=1$, $W(x)$ is continuous, increasing and concave on $[0,\infty)$,
\newline
$(W2)$ \hspace{1em} $W(x e^{x})=x$ and $ x=W(x)e^{W(x)}$,  $x\geq 0$,
\newline
$(W3)$ \hspace{1em} $W$ can be represented in the form of the absolutely convergent series
\begin{equation*}
W(x)=\ln x-\ln (\ln x)+\sum_{k=0}^{\infty}\sum_{m=1}^{\infty}c_{km}\frac{(\ln(\ln x))^m}{(\ln x)^{k+m}},\quad x\geq x_0>e,
\end{equation*}
with suitable constants $c_{km}$ and  $x_0 $, wherefrom  the following  estimates hold:
\begin{equation}
\label{sharpestimateLambert}
\ln x -\ln(\ln x)\leq W(x)\leq \ln x-\frac{1}{2}\ln (\ln x), \quad x\geq e.
\end{equation}
The equality in \eqref{sharpestimateLambert} holds if and only if $x=e$.

Note that $(W2)$ implies
\begin{equation*}
\label{PosledicaLambert1}
W(x\ln x)=\ln x,\quad x>1.
\end{equation*}
By using $(W3)$ we obtain
\begin{equation*}
\label{PosledicaLambert1.5}
W(x)\sim \ln x, \quad x\to \infty,
\end{equation*}
and therefore
\begin{equation*}
\label{PosledicaLambert2}
W(C x)\sim W(x),\quad x\to \infty,
\end{equation*}
for any $C>0$.
We refer to \cite{LambF, Mezo} for more details about the Lambert $W$ function.

%%%
%\subsection{Weight functions}  \label{subsec:weights}
%%%

%In particular, if $\omega$ is a weight function and $\omega_1\asymp\omega$ then
%\begin{equation}
%\label{OcenaYoung}
%A \varphi^*(y / A) \leq \varphi_1^*(y) \leq B\varphi^*(y/B)\quad y>0,
%\end{equation}
%for some $A,B>0$, where $\varphi (t) = \omega(e^t)$, $\varphi_1 (t) = \omega_1 (e^t)$ and $\varphi^*$, $\varphi_1^*$ are their Young's conjugates, respectively (see cite{MT}).

%%%
\subsection{Associated functions}
%%%

Let $(M_p)$ be an increasing sequence positive numbers which satisfies $(M.1)$, and $M_0 = 1$.
Then the {\em Carleman associated function } to the sequence $(M_p)$ is defined by
\begin{equation}
\label{associated-Katz}
\mu(h) = \inf_{p\in  \mathbb N} h^{-p} M_p, \quad h>0.
\end{equation}
This function is introduced in the study of quasi-analytic functions, see, e.g. \cite{Katznelson}.
We use the notation from \cite{GelfandShilov}.

In Komatsu's treatise of ultradistributions \cite{Komatsuultra1},
the associated function to $(M_p)$ is instead given by
\begin{equation}
\label{associated-Komatsu}
T(h) = \sup_{p>0} \ln_+ \frac{h^p }{M_p}, \quad h>0.
\end{equation}
%where $\ln_+ x=\max\{0,\ln x\}$, $x>0$.
%Then (since $M_p^{\tau,\sigma}$ satisfies $(M.1)$) we have
%\begin{equation} \label{NizFunkcija}
%M_p^{\tau,\sigma}=\sup_{k>0}\, k^p e^{-T_{\tau,\sigma}(k)},\qquad p\in \mathbb N.
%\end{equation}

\begin{Lemma} \label{lema:katz-komatsu}
Let $(M_p)$ be an increasing sequence positive numbers which satisfies $(M.1)$, and $M_0 = 1$, and let the functions $ \mu $ and $T$ be given by \eqref{associated-Katz} and \eqref{associated-Komatsu}
respectively. Then
\begin{equation}
\label{Katz-Komatsu}
\mu(h) = e^{-T(h)}, \quad h>0.
\end{equation}
\end{Lemma}

\begin{proof}
Clearly,
\begin{eqnarray*}
 T(h) =\sup_{p>0} \ln_+ (h^p M_p ^{-1}) &=&
 \sup_{p>0} (- \ln_+ (h^{-p} M_p) ) \\
& = &
 -\inf _{p>0} ( \ln_+ (h^{-p} M_p) ) =
 - \ln_+ (\inf _{p>0}  (h^{-p} M_p) ) \\
& = &  - \ln_+ \mu(h), \quad h>0,
 \end{eqnarray*}
which is  \eqref{Katz-Komatsu}.
\end{proof}

When $(M_p)$ is (equivalent to) the Gevrey sequence, $M_p = p^{sp}$, $p\in\mathbb N,$ $ s >1$,
an explicit calculation gives
\begin{equation*}
T (h) =  \frac{s}{e} h^{\frac{1}{s}}, \quad h>0.
\end{equation*}
Thus \eqref{Katz-Komatsu} implies that there exist constants $k>0$, and $C>0$ such that
\begin{equation*}
e^{-k h^{\frac{1}{s}}} \leq \mu(h) \leq C e^{-k h^{\frac{1}{s}}},  \quad h>0,
\end{equation*}
see also \cite[Ch IV, 2.1]{GelfandShilov}.

%When considering $M_p^{\tau ,\sigma}$ we get the following estimate.
%
%\begin{lemma} \label{asociatedFunctLema}
%Let there be given $\tau>0$ and $\sigma>1$. Then
%\begin{equation}\label{FunctionTtilda}
%{\mu}_{\tau,\sigma}(h):=
%\inf_{p>0}
%\frac{p^{\tau p^{\sigma}}}{ h^{p^{\sigma}}}=e^{-\frac{\tau}{\sigma e}h^{\frac{\sigma}{\tau}}}\quad, %h>0. \
%\end{equation}
%\end{lemma}
%
%\begin{proof}
%Put $f(p)=\frac{ p^{\tau p^{\sigma}}}{ h^{p^{\sigma}}}$, $p>0$.
%Since
%$$
%(\ln _+ f(p))'= p^{\sigma-1}(\tau\sigma \ln _+ p - \sigma\ln _+ h -\tau) = 0 \quad\text{for}
%\quad p_0:=h^{\frac{1}{\tau}}e^{-\frac{1}{\sigma}},
%$$
%we obtain
%$$
%\inf_{p>0}
%\ln _+  f(p) = \ln _+ f(p_0)= \frac{\tau}{e\sigma} h^{\frac{\sigma}{\tau}},
%$$
%and \eqref{FunctionTtilda} follows from \eqref{Katz-Komatsu} by replacing $M_p$ with $p^{\tau %p^{\sigma}}$ in
%\eqref{associated-Katz} and \eqref{associated-Komatsu}.
%\end{proof}

\par

By using \eqref{associated-Komatsu} we define the associated function to the sequence ${M_p^{\tau,\sigma}} = p^{\tau p^{\sigma}}$, $p\in\mathbb N,$ $\tau>0$, $\sigma>1$, as follows:
\begin{equation}\label{asocirana}
T_{\tau,\sigma}(h)=\sup_{p\in \mathbb N_0}\ln_+\frac{h^{p}}{M_p^{\tau,\sigma}},\quad \;\; h>0.
\end{equation}
%where $\ln_+ x=\max\{0,\ln x\}$, $x>0$.
%Then (since $M_p^{\tau,\sigma}$ satisfies $(M.1)$) we have
%\begin{equation} \label{NizFunkcija}
%M_p^{\tau,\sigma}=\sup_{k>0}\, k^p e^{-T_{\tau,\sigma}(k)},\qquad p\in \N_0.
%\end{equation}

It is a remarkable fact that $T_{\tau,\sigma}(h)$ can be expressed via the Lambert W function.

\begin{Theorem}
\label{TeoremaAsocirana}
Let $\tau>0$, $\sigma>1$,  $M^{\tau,\sigma}_p=p^{\tau p^\sigma}$, $ p \in \mathbb N$, and let $ T_{\tau,\sigma}(h)$ be given by \eqref{asocirana}.
Then
\begin{equation}
\label{AsociranaSigma}
 T_{\tau,\sigma}(h)\asymp \tau^{-\frac{1}{\sigma-1}}
 \frac{ \ln  ^{\frac{\sigma}{\sigma-1}}(h) }{ W^{\frac{1}{\sigma-1}}  (\ln(h))},
 \quad
 h \; \text{large enough},
\end{equation}
where the hidden constants depend on $\sigma$ only.
\end{Theorem}

\begin{proof}
The proof follows from \cite[Proposition 2]{TT0} and estimates (30) given in its proof.
More precisely, it can be shown  that
\begin{equation}
\label{KonacnaocenaAsocirana}
B_\sigma\,\tau^{-\frac{1}{\sigma-1}}\frac{ \ln^{\frac{\sigma}{\sigma-1}}(h) }{ W^{\frac{1}{\sigma-1}}  (\ln(h))}+\widetilde{B}_{\tau,\sigma}
\leq T_{\tau,\sigma}(h)
\leq A_{\sigma}\,\tau^{-\frac{1}{\sigma-1}}\frac{ \ln^{\frac{\sigma}{\sigma-1}}(h) }{ W^{\frac{1}{\sigma-1}}  (\ln(h))}+\widetilde{A}_{\tau,\sigma},\,\,
\end{equation}
for large enough $h>0$, and suitable constants $A_{\sigma}, B_{\sigma}, \widetilde{A}_{\tau,\sigma},\widetilde{B}_{\tau,\sigma}>0$.
\end{proof}

Since $T_{\tau,\sigma}(h)$, $h>0$, is an increasing function, \eqref{AsociranaSigma} implies that there exists $A>0$ such that
\begin{equation*}
\label{AlmostIncreasing}
\displaystyle
 \frac{ \ln  ^{\frac{\sigma}{\sigma-1}}(h) }{ W^{\frac{1}{\sigma-1}}  (\ln(h))}
 \leq A\,  \frac{ \ln  ^{\frac{\sigma}{\sigma-1}}(h+a) }{ W^{\frac{1}{\sigma-1}}  (\ln(h+a))},
 \quad a>0, \quad  h \; \text{large enough}.
\end{equation*}

We also notice that $(W3)$ (from subsection \ref{subsec:lambert})  implies
\begin{equation*}
\label{OcenaAsocirana}
T_{\tau,\sigma} (h)\asymp \left ( \frac{\ln^{\sigma }(h) }{\tau \ln (\ln (h))} \right )^{\frac{1}{\sigma -1} },
\quad \text{for} \quad  h \; \text{large enough}.
\end{equation*}

\par

%%%
\subsection{Associated function as a weight function}
%%%

%\begin{Remark}
The approach to ultradifferentiable functions via defining sequences is equivalent to the Braun-Meise-Taylor approach based on weight functions, when the defining sequences satisfy conditions $(M.1),$ $(M.2)$ and $(M.3),$ see \cite{Bonet,MT}. Since $M^{\tau,\sigma}_p=p^{\tau p^\sigma},$ $p\in\mathbb N,$ does not satisfy $(M.2),$ to compare the two approaches in \cite{TT0}, the authors used the technique of weighted matrices, see \cite{RS}. One of the main results from \cite{TT0} can be stated as follows.

\par

\begin{Proposition} \label{Tisaweight}
Let $\tau>0$, $\sigma>1$,  $M^{\tau,\sigma}_p=p^{\tau p^\sigma},$ $ p \in\mathbb N$, and let $ T_{\tau,\sigma}(h)$ be the associated function to the sequence $({M_p^{\tau,\sigma}})$. Then $ T_{\tau,\sigma}(h) \asymp \omega (h)$, where
$\omega $ is a weight function.
%, i.e. $\omega $ satisfies  ($\alpha$)--($\delta$).
\end{Proposition}

Recall, a \emph{weight function} is non-negative, continuous, even and increasing function defined on $\mathbb R_+ \cup \{ 0 \}$, $\omega(0)=0$,  if the following conditions hold:
\newline
($\alpha$) \hspace{1em} $\displaystyle \omega(2t)=O(\omega(t)),\quad t\to \infty,$
\newline
($\beta$) \hspace{1em} $\displaystyle \omega(t)=O(t),\quad t\to\infty,$
\newline
($\gamma$) \hspace{1em} $\displaystyle \ln t = o(\omega(t)), \quad t \to \infty,\;\;$ i.e.
$\displaystyle \lim_{t \to \infty} \frac{\ln t}{\omega(t)} = 0$,
\newline
($\delta$) \hspace{1em} $\displaystyle \varphi(t)=w(e^t)\quad$ is convex.

%The Young conjugate of the function $\varphi$ is given by
%\begin{equation}
%\label{Young}
%\varphi^*(k)=\sup_{t>0}(kt - \varphi(t)),\quad k\geq 0.
%\end{equation}

Some classical examples of weight functions are
\begin{equation*}
\label{BMTexamples}
\omega (t) = \ln^{s}_+ |t|,\quad\quad \omega (t)=\frac{|t|}{\ln^{s-1} (e+|t|)},\quad s>1,\, t\in \mathbb R,
\end{equation*}
where $\ln_+ x=\max\{0,\ln x\}$, $x>0.$ Moreover, $\omega (t)=|t|^s$ is a weight function if and only if $0<s\leq 1$.
%Note that by \eqref{PosledicaLambert1.5} it follows that $\omega(t)=W(|t|)$ is not a weight %function since the condition $(\gamma)$ is not satisfied.

%\begin{proof}
%We refer to \cite{TT0} for the proof of Proposition \ref{Tisaweight}.
%\end{proof}

We refer to \cite{RS} for the weighted matrices approach to ultradifferentiable functions.It is introduced in order to treat both
 Braun-Meise-Taylor and Komatsu methods in a unified way, see also subsection
 \ref{matrices-approach}.

%\end{Remark}

%%%%%
\section{Extended Gevrey regularity}
\label{SectionSpaces}
%%%%%

%%
\subsection{Extended Gevrey classes and their dual spaces}

Recall that the Gevrey space ${\mathcal G}_t ( \mathbb{R}^d)$, $ t >1$, consists of functions $  \phi \in C^\infty ( \mathbb{R}^d)$ such that for every compact set   $K\subset \subset \mathbb{R}^d$ there are constants
$h>0$ and $C_K > 0 $ satisfying
\begin{equation}\label{Eq:Gevrey-class}
|\partial^{\alpha} \phi (x)| \leq  C_K h^{|\alpha|} |\alpha|!^t,
\end{equation}
for all $x \in K $ and for all $\alpha \in \mathbb N_0^d$.

In a similar fashion we introduce new  classes of smooth functions by using defining sequences $M^{\tau,\sigma}_p=p^{\tau p^\sigma}$, $ p \in \mathbb{N} $, $\tau>0$, $\sigma>1$.

\begin{Definition}
\label{def:extendedGervey}
Let there  be given $\tau>0$, $\sigma>1$, and let $M^{\tau,\sigma}_p=p^{\tau p^\sigma}$, $ p \in \mathbb{N} $, $M^{\tau,\sigma}_0 = 1$.

The extended Gevrey class of Roumieu type ${\mathcal E}_{\{\tau, \sigma\}}( \mathbb{R}^d)$ is the set of all
$\phi \in  C^{\infty}( \mathbb{R}^d)$ such that for every compact set   $K\subset \subset \mathbb{R}^d$ there are constants $h > 0 $ and $C_K >0$ satisfying
\begin{equation} \label{NewClassesInd}
|\partial^{\alpha} \phi (x)| \leq C_K h^{|\alpha|^{\sigma}}  M_{|\alpha|} ^{\tau,\sigma},
\end{equation}
for all $x \in K $ and for all $ \alpha \in  \mathbb{N}_0 ^d.$

The extended Gevrey class of Beurling type ${\mathcal E}_{(\tau, \sigma)}( \mathbb{R}^d)$ is the set of all
$\phi \in  C^{\infty}( \mathbb{R}^d)$ such that
for every compact set   $K\subset \subset \mathbb{R}^d$ and for all $h>0$ there is a constant $C_{K,h} > 0 $  satisfying
\begin{equation} \label{NewClassesProj}
|\partial^{\alpha} \phi (x)| \leq C_{K,h} h^{|\alpha|^{\sigma}}  M_{|\alpha|} ^{\tau,\sigma},
\end{equation}
for all $x \in K $ and for all $ \alpha \in  \mathbb{N}_0^d.$
\end{Definition}

The spaces  ${\mathcal E}_{\{\tau, \sigma\}}( \mathbb{R}^d)$ and ${\mathcal E}_{(\tau, \sigma)}( \mathbb{R}^d)$ are in a usual way endowed with projective and inductive limit topologies respectively, we refer to \cite{PTT-01} for details. In particular, they are nuclear spaces, see
\cite[Theorem 3.1]{PTT-01}.

Note that \eqref{Eq:Gevrey-class},  \eqref{NewClassesInd}  and \eqref{NewClassesProj} imply
\begin{equation*}
    \cup_{\tau > 1} {\mathcal G}_t ( \mathbb{R}^d)  \hookrightarrow
{\mathcal E}_{(\tau, \sigma)}( \mathbb{R}^d) \hookrightarrow {\mathcal E}_{\{\tau, \sigma\}}( \mathbb{R}^d),
\end{equation*}
where $\hookrightarrow$ denotes continuous and dense inclusion.

\par

The set of functions $ \phi \in {\mathcal E}_{\{\tau, \sigma\}}( \mathbb{R}^d)$
($ \phi \in {\mathcal E}_{(\tau, \sigma)}( \mathbb{R}^d)$) whose support is contained in some compact set is denoted by ${\mathcal D}_{\{\tau, \sigma\}} ( \mathbb{R}^d)$
( ${\mathcal D}_{(\tau, \sigma)} ( \mathbb{R}^d)$ ).
\par

We use the abbreviated notation $ \tau,\sigma $ for $\{\tau,\sigma\}$ or $(\tau,\sigma)$ to denote ${\mathcal E}_{\tau, \sigma}( \mathbb{R}^d) = {\mathcal E}_{\{\tau,\sigma\}}( \mathbb{R}^d)$ or ${\mathcal E}_{\tau, \sigma}( \mathbb{R}^d)  = {\mathcal E}_{(\tau, \sigma)}( \mathbb{R}^d)$,
and similarly for ${\mathcal D}_{\tau, \sigma} ( \mathbb{R}^d)$.

Next we give an equivalent description of extended Gevrey classes by using sequences from
$\mathcal{R}$, see subsection \ref{subsec:sequences}. We note that such descriptions are important when dealing with integral transforms of ultradifferentiable functions and
related ultradistributions,
cf. \cite{PilipovicKnjiga, Komatsuultra3,P-13}.
The result follows from a lemma which is a modification of  \cite[Lemma 3.4]{Komatsuultra3}, and  \cite[Lemma 2.2.1]{PilipovicKnjiga}.

Put $\lfloor x \rfloor :=\max\{m\in \mathbb N\,:\,m\leq x\} $ (the greatest integer part of  $x\in \mathbb R_+$).

\begin{Lemma} \label{lm:konstante-i-nizovi}
Let there be given $\sigma > 1$,  a sequence of positive numbers $(a_p)$,
$\left(r_j\right) \in  \mathcal{R}$, and put
\begin{equation}
\label{proizvod}
R_{0,\sigma}=1, \quad R_{p,\sigma}:=\prod_{j=1}^{\lfloor p^{\sigma} \rfloor} r_j, \quad p \in \mathbb{N}.
\end{equation}
Then the following is true.
\label{LemaProizvodNizova}
\begin{itemize}
\item[i)] There exists $h>0$ such that
\begin{equation*}
\label{pre1}
\sup \left\{\frac{a_{p}}{h^{p^{\sigma}}}: p \in \mathbb{N}_0\right\}<\infty,
\end{equation*}
if and only if
\begin{equation}
\label{pre2}
\sup \left\{\frac{a_{p}}{R_{p,\sigma}}: p \in \mathbb{N}_0\right\}<\infty,
\quad
\text{for any } \quad \left(r_j\right) \in  \mathcal{R}.
\end{equation}

\item[ii)] There exists $\left(r_j\right) \in  \mathcal{R}$ such that
\begin{equation*}
\label{pre3}
\sup \left\{R_{p,\sigma} a_{p}:  p \in \mathbb{N}_0\right\}<\infty,
\end{equation*}
if and only if
\begin{equation}
\label{pre4}
\sup \left\{h^{p^{\sigma}} a_{p}: p \in \mathbb{N}_0\right\}<\infty,
\quad
\text{for every } \quad h>0.
\end{equation}
\end{itemize}
\end{Lemma}

The proof of Lemma \ref{lm:konstante-i-nizovi} is given in the Appendix.

Note that in \eqref{NewClassesInd} and \eqref{NewClassesProj} we could put
$h^{\lfloor|\alpha|^{\sigma} \rfloor}$ instead of $h^{|\alpha|^{\sigma}}$
(this follows from the simple inequality $\lfloor p^{\sigma}\rfloor\leq p^{\sigma}\leq  2\lfloor p^{\sigma}\rfloor$, $p\in {\mathbb N}$).

\begin{Proposition}
\label{prop:extendedGervey-nizovi}
Let there  be given $\tau>0$, $\sigma>1$, and let $M^{\tau,\sigma}_p=p^{\tau p^\sigma}$, $ p \in \mathbb{N} $, $M^{\tau,\sigma}_0 = 1$.
Then the following is true:
\begin{itemize}
    \item[i)] $ \phi \in  {\mathcal E}_{\{\tau, \sigma\}}( \mathbb{R}^d)$ if and only if
for every compact set   $K\subset \subset \mathbb{R}^d$,
and for any $\left(r_p\right) \in  \mathcal{R}$ and
$R_{p,\sigma}$ given by \eqref{proizvod},
there exists $C_{K, (r_p)} >0$ such that
\begin{equation*} \label{NewClassesInd-nizovi}
|\partial^{\alpha} \phi (x)| \leq
C_{K, (r_p)}
%\prod_{j=1}^{\lfloor |\alpha|^{\sigma} \rfloor} r_j
R_{\lfloor |\alpha|^{\sigma} \rfloor, \sigma}
M_{|\alpha|} ^{\tau,\sigma},
\end{equation*}
for all $x \in K $, and all $ \alpha \in  \mathbb{N}_0 ^d$.
 \item[ii)] $ \phi \in {\mathcal E}_{(\tau, \sigma)}( \mathbb{R}^d)$  if and only if
for every compact set  $K\subset \subset \mathbb{R}^d$ there is a sequence
$\left(r_p\right) \in  \mathcal{R}$ and a constant $C_{K} > 0 $  satisfying
\begin{equation*} \label{NewClassesProj-nizovi}
|\partial^{\alpha} \phi (x)| \leq C_{K} \frac{M_{|\alpha|} ^{\tau,\sigma}}{R_{\lfloor |\alpha|^{\sigma} \rfloor, \sigma}}
%{\prod_{j=1}^{\lfloor |\alpha|^{\sigma} \rfloor} r_j},
\end{equation*}
for all $x \in K $ and for all $ \alpha \in  \mathbb{N}_0^d,$
where $R_{p,\sigma}$ given by \eqref{proizvod}.
\end{itemize}
\end{Proposition}

Proposition \ref{prop:extendedGervey-nizovi} follows from Lemma  \ref{lm:konstante-i-nizovi}.

%%%%%%%%%%%%%%%%%%%%%%%%% ultradistributions

We end this subsection by introducing spaces of ulradistributions as dual spaces of ${\mathcal E}_{\tau, \sigma}( \mathbb{R}^d)$, $\tau>0$, $\sigma>1$. In subsection \ref{subsec:peliviner} we will prove a Paley-Wiener type theorem for such ultradistributions.

\begin{Definition}
\label{def:extendedGerveydual}
Let $\tau>0$ and $\sigma>1$, and let  ${\mathcal E}_{\tau, \sigma}( \mathbb{R}^d) = {\mathcal E}_{\{\tau,\sigma\}}( \mathbb{R}^d)$ or ${\mathcal E}_{\tau, \sigma}( \mathbb{R}^d)  = {\mathcal E}_{(\tau, \sigma)}( \mathbb{R}^d)$.
Then $u\in \mathcal E'_{\tau,\sigma}(\mathbb R^d)$ if there exists a compact set $K$ in $\mathbb R^d$ and constants $\varepsilon,$ $ C>0$ such that
\begin{equation}
\label{ultraNejednakost}
| ( u,\varphi )|\leq C \sup_{\alpha\in \mathbb N^d, x\in K}\frac{|D^{\alpha}\varphi(x)|}{\varepsilon^{|\alpha|^{\sigma}} |\alpha|^{\tau |\alpha|^{\sigma}}},\quad\forall \varphi\in \mathcal E_{\tau,\sigma}(\mathbb R^d),
\end{equation}
and $( \cdot,\cdot )$ denotes standard dual pairing.

In a similar way ${\mathcal D}_{\tau, \sigma} ' ( \mathbb{R}^d)$
is the dual space of  ${\mathcal D}_{\tau, \sigma} ( \mathbb{R}^d)$.
\end{Definition}

%%%%%%%%%%

%%
\subsection{Example of a compactly supported function}

The non-quasianalyticity condition $(M.3)'$ provides the existence of nontrivial compactly supported functions in $\mathcal E_{\tau,\sigma}(\mathbb{R}^d)$ which can be formulated as follows.

\begin{Proposition} \label{NasaCutoff}
Let $\tau>0$ and $\sigma>1$.  For every $a>0$ there exists
$\phi_a \in {\mathcal E_{\tau,\sigma}}(\mathbb{R}^d)$ such that $\phi_a  \geq 0$, $\supp\phi_a \subset [-a,a]^d$, and $\int_{\mathbb R ^d} \phi_a (x) \, dx = 1$.
\end{Proposition}

Of course, any compactly supported Gevrey function from $\mathcal G_{\tau}(\mathbb{R}^d)$ will suffice. However,  the construction in Proposition \ref{NasaCutoff},  is sharp in the sense that $\phi_a$ does not belong to any Gevrey class, i.e.
$\phi_a \not\in \bigcup_{t>1}  {\mathcal G}_t ({\mathbb R}^d)$.
% In fact, the regularity of $\phi $ is related to the minimal size of its support.
We refer to the proof of  \cite[Lemma 1.3.6.]{H} for more details.

\begin{proof} We give a proof when $d=1$, and for $d\geq 2$ the proof follows by taking the tensor product.

Since $ \mathcal{D}_{\tau,  \sigma}({\mathbb R})$ is closed under dilation and multiplication by a constant, it is enough to show  the result for $a=1$, and set $\phi_1 = \phi$.

From
\begin{equation*}
  \sum_{p=1}^{\infty} \frac{1}{(2 (p + 1))^{\frac{1}{m}p^{\sigma - 1}}}< \infty
\end{equation*}
for any $m\in\mathbb N _0$ and any given $\sigma >1$,
it follows that there exists a sequence of nonnegative integers $(N_m)$ such that
\begin{equation*}
    \sum_{n=N_m}^{\infty} \frac{1}{(2 (p + 1))^{\frac{1}{m}p^{\sigma - 1}}} < \frac{1}{2^m}.
\end{equation*}
Thus the sequence  $a_p $, $p \in \mathbb N _0$, given by
\begin{equation*}
    a_p := \frac{1}{(2 (p + 1))^{\frac{1}{m}p^{\sigma - 1}}}, \quad N_m  \leq p < N_{m+1},
\end{equation*}
satisfies
\begin{equation*}
    \sum_{n=N_1}^{\infty} a_p \leq 1.
\end{equation*}

Let $f\in C^{\infty} (\mathbb R )$ be a non-negative and even function such that
$\supp f \in [-1,1]$, $\int_{-1}^1 f(x) \, dx = 1$,
and $f_{a} (x) = \frac{1}{a} f\left(\frac{x}{a}\right)$.
Then we define the sequence of functions $(\phi_p)$ by
\begin{equation*}
    \phi_{p} := f_{a_{N_1}}* f_{a_{N_1 + 1}}*\dots * f_{a_{p}}, \qquad p \in   \mathbb N _0.
\end{equation*}
Note that
\begin{equation*}
    \supp\phi_{p} \subset [-1,1], \qquad
\int_{-1}^1 \phi_{p}(x) \, dx = 1, \qquad
p\in\mathbb N _0,
\end{equation*}
\begin{equation*}
    \phi_{p} ^{(n)} = f_{a_{N_1}}*\dots * f_{a_{N_m}}* f'_{a_{N_m + 1}}*\dots * f '_{a_{N_m + n}}*
f_{a_{N_m + n + 1}}*\dots * f_{a_{p}},
\end{equation*}
and
\begin{equation} \label{procena-f}
\| f' _{a_p} \|_1 = \frac{1}{a_p}
\int_{\mathbb{R}} \frac{1}{a_p} \left| f' \left( \frac{x}{a_p} \right) \right | \, dx
\leq \frac{c}{a_p} \leq c \, (2 (p + 1))^{\frac{1}{m}p^{\sigma - 1}},
\end{equation}
when $p\geq N_m$.

Let there be given $n\in\mathbb N _0$ and $\tau > 0$. Then we choose $m,p\in\mathbb N _0$ so that $1/m < \tau$, and $N_m + n < p$.

By using \eqref{procena-f} and the fact that ${\widetilde{(M.2)'}}$ implies
\begin{equation*}
    M^{\frac{1}{m}, \sigma} _{p+q} \leq \tilde{C}^{p^\sigma} M^{\frac{1}{m}, \sigma} _{p}
\end{equation*}
for some $\tilde{C}= \tilde{C}(q)>0$, we obtain
\begin{align*}
|\phi_p ^{(n)} (x)| &\leq c^n \, 2^{\frac{1}{m} \sum_{k=1}^n (N_m+k)^{\sigma - 1}} \prod_{k=1}^n (N_m + k + 1)^{\frac{1}{m} (N_m+k)^{\sigma - 1}}\\
&\leq c^n 2^{\frac{1}{m} (N_m + n)^{\sigma }} (N_m + n + 1)^{\frac{1}{m} (N_m + n + 1)^{\sigma }}\\
&
\leq C^{n^{\sigma}} n^{\tau n^{\sigma}},
\end{align*}
where  $C$ depends on $\tau$.

The sequence $ \{ \phi_p ^{(n)} \mid  p = N_1, N_2, \dots \} $ is a Cauchy sequence for every $ n \in\mathbb N_0$. Thus, it converges to a function $ \phi $ that satisfies
\begin{equation*}
    |\phi ^{(n)} (x)| \leq C^{n^{\sigma}} n^{\tau n^{\sigma}},
\end{equation*}
for every $\tau > 0$. Therefore, $\phi \in \mathcal{D}_{\tau, \sigma}({\mathbb R})$, and by the construction $\phi  \geq 0$, $\supp\phi \subset [-1,1]$ and $\int_{\mathbb R } \phi (x) \, dx = 1$, which completes the proof.
\end{proof}

\subsection{Algebra property}

Since $M_p^{\tau,\sigma}$ satisfies properties $(M.1)$ and $\widetilde{(M.2)'}$
we have the following.

\begin{Proposition}
\label{propalgebra}
${\mathcal E}_{\tau, \sigma}(\mathbb R^d)$ is closed under the pointwise multiplication of functions
and under the (finite order) differentiation.
\end{Proposition}

\begin{proof}
Let us prove that ${\mathcal E}_{(\tau, \sigma)}(\mathbb R^d)$ is closed under pointwise multiplication, since its closedness under the differentiation follows from $\widetilde{(M.2)'}$.
We refer to \cite{PTT-01} for the Roumieu case ${\mathcal E}_{\{\tau, \sigma\}}(\mathbb R^d)$ .

Let $\phi,\psi\in {\mathcal E}_{(\tau, \sigma)}(\mathbb R^d)$. Let $K$ be a compact subset of
$\mathbb R^d$. Then for every $h,k>0$ there exist constants $ C_{K,h}>0$ and $C_{K,k}>0$
such that
\begin{equation*}
\sup_{x\in K}
|\partial^{\alpha} \phi (x)| \leq  C_{K,h} h^{|\alpha|^{\sigma}}|\alpha|^{\tau|\alpha|^{\sigma}}, \quad
\text{and} \quad
\sup_{x\in K}
|\partial^{\alpha} \psi (x)| \leq \tilde C_{K,k} k^{|\alpha|^{\sigma}}|\alpha|^{\tau|\alpha|^{\sigma}}<\infty.
\end{equation*}
For simplicity, assume that $\tau = 1$, and the proof for $\tau > 1$ is similar.

\par

By the Leibniz formula and  $(M.1)$ we have
\begin{multline*}
    |\partial^{\alpha}(\phi \psi)(x)| \leq  \sum_{\beta \leq \alpha} \binom{\alpha}{\beta} |\partial^{\alpha-\beta}\phi(x)||\partial^{\beta}\psi(x)| \\
    \leq  C_{K,h} \tilde C_{K,k} \sum_{\beta \leq \alpha} \binom{\alpha}{\beta} h^{|\alpha-\beta|^{\sigma}}|\alpha - \beta|^{|\alpha-\beta|^{\sigma}}k^{|\beta|^{\sigma}}|\beta|^{|\beta|^{\sigma}}\\
\leq  C_{K,h} \tilde C_{K,k}  |\alpha|^{|\alpha|^{\sigma}}
\sum_{\beta \leq \alpha} \binom{\alpha}{\beta} h^{|\alpha-\beta|^{\sigma}}k^{|\beta|^{\sigma}}, \;\;\; x \in K.
\end{multline*}

By choosing $h=k$ we get
\begin{equation*}
    \sum_{\beta \leq \alpha} \binom{\alpha}{\beta} h^{|\alpha-\beta|^{\sigma}}k^{|\beta|^{\sigma}}
\leq 2^{|\alpha|} h^{2^{\sigma} |\alpha|^{\sigma}} \leq
(2^{|\alpha|} h^{2^{\sigma}})^{ |\alpha|^{\sigma}},
\end{equation*}
and obtain
\begin{equation*}
      |\partial^{\alpha}(\phi \psi)(x)|\leq   C (2^{|\alpha|} h^{2^{\sigma}})^{ |\alpha|^{\sigma}} |\alpha|^{|\alpha|^{\sigma}},
\end{equation*}
with $ C=  C_{K,h} \tilde C_{K,h}$.

Thus, for any given $\tilde h>0$ we can choose $h <(\tilde h/2)^{1/2^{\sigma}}$
to get
\begin{equation*}
      |\partial^{\alpha}(\phi \psi)(x)|\leq   C\tilde h^{ |\alpha|^{\sigma}} |\alpha|^{|\alpha|^{\sigma}},
\end{equation*}
where $C >0$ depends on $K$ and $\tilde h$, that is,
$ \phi \psi \in {\mathcal E}_{(\tau, \sigma)}(\mathbb R^d)$.
%%%%%%%%%%%%
\end{proof}

\subsection{Inverse closedness and composition}

We need some preparation related to the decompositions that appear when using the
generalized Fa\`a di Bruno formula.

Let there be given a multiindex $\alpha\in \mathbb N^d$.  We say that $\alpha $  is decomposed into parts $p_1,\dots,p_s\in \mathbb N^d$ with multiplicities $m_1,\dots,m_s\in \mathbb N$, if
\begin{equation}
\label{oblikParticijaFaa}
\alpha=m_1 p_1+m_2 p_2+\dots+m_s p_s,
\end{equation}
where $|p_i|\in\{1,\dots,|\alpha|\}$, $m_i\in\{0,1,\dots, |\alpha|\}$, $i =1,\dots,s$.
If $p_i=(p_{i_1},\dots, p_{i_d})$, $i\in \{1,\dots,s\}$, we put $p_i<p_j$ when $i<j$ if there exists $k\in \{1,\dots,d\}$ such that $p_{i_1}=p_{j_1},\dots, p_{i_{k-1}}=p_{j_{k-1}}$ and $p_{i_{k}}<p_{j_{k}}$.
Note that $s\leq |\alpha|$ and $m= m_1+\dots+m_s \leq |\alpha|$.

The triple  $(s,p,m)$ is called the decomposition of $\alpha$ and the set of all decompositions of the form \eqref{oblikParticijaFaa} is denoted by $\pi$.

For smooth functions $f: \mathbb R\to {\mathbb C}$ and $g: \mathbb R^d\to \mathbb R$, the generalized Fa\`a di Bruno formula is given by
\begin{equation}
\label{GeneralizedFaaDiBruno}
\partial^{\alpha}(f(g))=\alpha! \sum_{(s,p,m)\in \pi} f^{(m)}(g)\prod_{k=1}^{s}\frac{1}{m_k!}\Big(\frac{1}{p_k !}\partial^{p_k} g \Big)^{m_k}.
\end{equation}

The total number ${\rm card}\, \pi$ of different decompositions of a multiindex $\alpha\in \mathbb N^d$ given by \eqref{oblikParticijaFaa}  can be estimated as follows:
$ \displaystyle {\rm card}\, \pi\leq (1+|\alpha|)^{d+2},$
cf. \cite[Remark 2.2]{TT-01}.

\begin{Theorem}
\label{teoremaInverseClosedness}
Let  $\tau>0$, $\sigma>1$. Then the extended Gevrey class $\mathcal E_{\tau,\sigma}(\mathbb R^d)$ is inverse-closed in $C^{\infty}(\mathbb R^d)$.
\end{Theorem}

\begin{proof}
The proof for $\mathcal E_{\{\tau,\sigma\}}(\mathbb R^d)$ is given in \cite{TT-01}. Here we give the proof for $\mathcal E_{(\tau,\sigma)}(\mathbb R^d).$

\par

Let $\phi\in \mathcal E_{(\tau,\sigma)}(\mathbb R^d)$, $\phi(x)\not=0$, $x\in \mathbb R^d$, and let $K$ be a compact set in $\mathbb R^d$.

When $d=1$ the proof is straightforward: Let $c \in (0,1) $ be such that
$ |\phi (x)| \geq c$. Then we have
\begin{equation*}
     \left|\left(\frac{1}{\phi(x)}\right)^{(\alpha)} \right |  \leq
  \frac{\alpha !}{|\phi(x)|^{\alpha+1}} |\phi^{(\alpha)}(x)|
\leq  C_{K,h} \alpha ! \left( \frac{1}{c} \right) h^{\alpha^{\sigma}} M^{\tau,\sigma} _{\alpha},
\end{equation*}
so for each $\tilde h >0$ there is $C>0$ such that
\begin{equation*}
      \left |\left(\frac{1}{\phi(x)}\right)^{(\alpha)}\ \right |
\leq C (\tilde h)^{\alpha^{\sigma}} M^{\tau,\sigma} _{\alpha},
\end{equation*}
that is, $1/\phi \in  \mathcal E_{(\tau,\sigma)}(\mathbb R^d)$.

When $d\geq 2$, we employ the Fa\`a di Bruno formula  \eqref{GeneralizedFaaDiBruno},  to obtain
\begin{equation*}
\label{GeneralizedFaaDiBrunoInverseIneq}
\left |\partial^{\alpha}\Big(\frac{1}{\phi(x)}\Big) \right |
\leq |\alpha|!\sum_{(s,p,m)\in \pi}\frac{m!}{m_1!\dots m_s! |\phi(x)|^{m+1}}\prod_{k=1}^{s}\Big(\frac{|\partial^{p_k} \phi(x)| }{p_k!}\Big)^{m_k},
\end{equation*}
for arbitrary $x\in K$.
Let $c \in (0,1)$ be chosen so that $|\phi(x)|\geq c$ for $x\in K$. Then
\begin{equation*}
\label{GeneralizedFaaDiBrunoInverse1}
\frac{m!}{m_1!\dots m_s! |\phi(x)|^{m+1}}\leq
\frac{s^m}{c^{m+1}} \frac{ m_1!\dots m_s!}{m_1!\dots m_s!}
\leq C^{|\alpha|^{\sigma}+1} ,
\end{equation*}
for a suitable constant $C>0$, where we used  $s,m\leq |\alpha|$, and $\sigma > 1$.

It remains to show that for each $\tilde h>0$ there exists $C>0$ such that
\begin{equation}
\label{productestimate}
\prod_{k=1}^{s}\Big(\frac{|\partial^{p_k} \phi(x)| }{p_k!}\Big)^{m_k}
\leq  C \tilde h ^{|\alpha|^{\sigma}} M^{\tau,\sigma} _{|\alpha|}.
\end{equation}
This can be done by induction with respect to the length of the multi-index $\alpha \in \mathbb N^d$. The proof for  $|\alpha| = 1$ is the same as in $d=1$.
Now, if \eqref{productestimate} holds for  $ |\alpha| <n$, the case $ |\alpha| = n$, follows from the induction step and Proposition \ref{propalgebra}. We omit details.
\end{proof}

\begin{Theorem}
\label{TeoremaKompozicija}
Let $\tau>0$, $\sigma>1$. If $f\in \mathcal E_{\tau,\sigma}(\mathbb R)$ and $g\in \mathcal E_{\tau,\sigma}(\mathbb R^d)$ is such that $g: \mathbb R^d\to \mathbb R$, then $f\circ g\in \mathcal E_{\tau,\sigma}(\mathbb R^d)$.
\end{Theorem}

The proof of Theorem \ref{TeoremaKompozicija} for (the Roumieu case) can be found in \cite{PTT-02}.
Theorem \ref{teoremaInverseClosedness} is a consequence of Theorem \ref{TeoremaKompozicija},
but, as we see, it can be proved independently.

\subsection{Paley-Wiener theorems} \label{subsec:peliviner}

Let $\mathcal E_{(\sigma)}(\mathbb R^d)=\displaystyle\bigcap_{\tau>0}\mathcal E_{(\tau,\sigma)}(\mathbb R^d)$ and let
$\mathcal D_{(\sigma)}(\mathbb R^d)$ denote the set of
compactly supported elements from $\mathcal E_{(\sigma)}(\mathbb R^d).$

A more general statement than Proposition \ref{NASPW} is given in \cite[Theorem 3.1]{PTT-03}.

\begin{Proposition} \label{NASPW}
Let $\sigma > 1$, and let $f\in  \mathcal D_{( \sigma ) } ({\mathbb R^d}).$
Then $\widehat{f}$, the Fourier transform of  $f$, is analytic function, and for every $h>0$ there exists a constant $C_{h}>0$ such that
\begin{align}
\label{ocenaTeorema1}
|\widehat{f}(\xi)| \leq C_{h} \,  \exp\left\{ - h \left (
 \ln^{\frac{\sigma}{\sigma-1}}(|\xi|) / W^{\frac{1}{\sigma-1}}  (\ln(|\xi|))
\right ) \right\}, \quad |\xi| \quad \text{large enough},
\end{align}
where $W$ denotes the Lambert function.
\end{Proposition}

%Notice that for $\sigma = 2 $ \eqref{ocenaTeorema1} amounts to \eqref{Eq:Lambert-decay} (with $\psi =\widehat{f}$).

\begin{proof} The analyticity of $f$ follows from the classical Paley-Wiener theorem, cf. \cite{H}. It remains to prove \eqref{ocenaTeorema1}.

%Take arbitrary $h>0$ and
Let  $f\in \, \mathcal D_{(\sigma)} ({\mathbb R^d})$, and let $K$ denote the support of $f$.
Since $f\in {\mathcal D_{\frac{\tau}{2}, \sigma}}({\mathbb R^d})$, by
Definition \ref{def:extendedGervey} for every
$ \alpha\in \mathbb N^d$ we get the following estimate:
\begin{equation*}
    |\xi^\alpha \widehat f(\xi)|=  |\widehat{D^{\alpha} f}(\xi)|\leq C   \sup_{x\in K}|{D^{\alpha} f}(x)|\leq C^{|\alpha|^\sigma +1}_1 |\alpha|^{\frac{\tau}{2} |\alpha|^\sigma}\leq C_2 |\alpha|^{\tau |\alpha|^{\sigma}},\, \xi\in \mathbb R^d,
\end{equation*}
for a suitable constant $C_2>0$. Now, the relation between the sequence $(M_p^{\tau,\sigma})$ and its associated function
$ T_{\tau,\sigma}$ given by \eqref{asocirana} implies that
\begin{equation*}
\label{PWsaT}
|\widehat f(\xi)|
\leq C_2 \inf_{\alpha\in \mathbb N^d}\frac{|\alpha|^{\tau |\alpha|^\sigma}}{|\xi|^{|{\alpha}|}} \leq  C_3 e^{- T_{\tau,\sigma}(|\xi|)},\quad \quad |\xi| \quad \text{large enough},
\end{equation*} for suitable $C_3>0$.
%, where for the last inequality we use that $|\eta^{\alpha}|\geq C |\eta|^{\alpha}$ for some $C>0$.
Then, from the left-hand side of \eqref{KonacnaocenaAsocirana} we get
\begin{equation*}
    |\widehat f(\xi)| \leq C_4 \exp\left\{- B_\sigma\,\tau^{-\frac{1}{\sigma-1}}
\left ( \ln^{\frac{\sigma}{\sigma-1}}(|\xi|) / W^{\frac{1}{\sigma-1}}  (\ln(|\xi|)) \right ) \right\},
\quad |\xi| \quad \text{large enough},
\end{equation*}
with $\displaystyle C_4 =  C_3 e^{-\widetilde{B}_{\tau,\sigma}} $. For any given $h>0$ we choose
$ \tau=(B_{\sigma}/h)^{\sigma-1}$, to obtain \eqref{ocenaTeorema1}, which proves the claim.
\end{proof}

We proceed with the Paley-Wiener theorem for $u\in\mathcal E'_{(\sigma)}(\mathbb R^d)$.
%extended Gevrey ultradistributions with compact support. Such ultradistributions are continuous %functionals over $\mathcal E_{\tau,\sigma}(\mathbb R^d)$. We will
%discuss  $\mathcal E' _{(\tau,\sigma)}(\mathbb R^d)$ since they are larger class than $\mathcal E' _{\{\tau,\sigma \}}(\mathbb R^d)$.
%
% We note that for fixed $\tau_0>0$ we have
%$$\mathcal E'_{(\tau_0,\sigma)}(\mathbb R^d)\hookrightarrow \mathcal E'_{(\sigma)}(\mathbb R^d)=\bigcup_{\tau>0}\mathcal E'_{(\tau,\sigma)}(\mathbb R^d),$%$ where $\hookrightarrow$ denotes continuous embedding with respect to strong topology.
%

\begin{Theorem} \label{Thm:Peli-Viner}
Let $\sigma>1$.

\begin{enumerate}

\item[i)] If $u\in\mathcal E'_{(\sigma)}(\mathbb R^d)$ then there exist constants $h, C>0$ such that

\begin{equation*}
|\widehat u (\xi)|\leq C\exp\left\{h \Big(\frac{\ln^\sigma |\xi|}{
W (\ln(|\xi|))}\Big)^{\frac{1}{\sigma-1}}\right\},\quad \textit{for $|\xi|$ large enough}.
\end{equation*}

\item[ii)] If $u\in\mathcal E'_{(\sigma)}(\mathbb R^d)$ and if for every $h>0$ there exists $C>0$ such that
\begin{equation}
\label{NejednakostPaley}
|\widehat u (\xi)|\leq C\exp\left\{-h \Big(\frac{\ln^\sigma (|\xi|)}{ W (\ln(|\xi|))}\Big)^{\frac{1}{\sigma-1}}\right\},\quad \textit{for $|\xi|$ large enough},
\end{equation} then $u\in \mathcal E_{(\sigma)}(\mathbb R^d)$.
\end{enumerate}
\end{Theorem}

\begin{proof}
$i)$ Fix $\tau_0>0$ so that $u\in \mathcal E' _{(2\tau_0,\sigma)}(\mathbb R^d)$. By applying \eqref{ultraNejednakost} to $\varphi_\xi (x)=e^{-2\pi i x\cdot\xi} \in \mathcal E_{(\sigma)}(\mathbb R^d)$, $\xi\in \mathbb R^d$, we get
\begin{multline*}
%\label{ultraNejednakost}
|\widehat u (\xi)|=|( u, e^{-2\pi i \cdot \xi})|
\leq C \sup_{\alpha\in \mathbb N^d}\sup_{x\in K}
\frac{|D^{\alpha}(e^{-2\pi i x\cdot\xi})|}{\varepsilon^{|\alpha|^{\sigma}} |\alpha|^{2\tau_0 |\alpha|^{\sigma}}} \\
\leq C \sup_{\alpha\in \mathbb N^d}\frac{|\xi|^{|\alpha|}}{ \varepsilon^{|\alpha|^{\sigma}}|\alpha|^{2\tau_0 |\alpha|^{\sigma}}}\leq C_1 \sup_{\alpha\in \mathbb N^d}\frac{|\xi|^{|\alpha|}}{ |\alpha|^{\tau_0 |\alpha|^{\sigma}}}=C_1\exp\{ T_{\tau_0,\sigma}(|\xi|)\},\quad \xi\in\mathbb R^d,
\end{multline*}
where we have used simple inequalities $|\xi^{\alpha}|\leq |\xi|^{|\alpha|}$ and $\varepsilon^{|\alpha|^{\sigma}}|\alpha|^{2\tau_0 |\alpha|^{\sigma}}\geq C' |\alpha|^{\tau_0 |\alpha|^{\sigma}}$ for suitable $C'>0$. Now the the statement follows from \eqref{KonacnaocenaAsocirana}.

$ii)$ % Let $K$ be the support of $u \in \in\mathcal E'_{(\sigma)}(\mathbb R^d)$.
%Note that the assumptions imply that $\hat u$ is smooth and integrable function. Moreover,
It is sufficient to prove that for every $\tau>0$ there exists constant $C>0$ such that
\begin{equation} \label{strongerestimate}
    \sup_{x\in  \mathbb R^d}|D^{\alpha}u(x)|\leq C |\alpha|^{\tau |\alpha|^{\sigma}},
\qquad \alpha\in \mathbb N^d.
\end{equation}

Notice that $h>0$ large enough in \eqref{KonacnaocenaAsocirana} can be replaced by
$(1+|\xi|)$ for all $\xi \in \mathbb R^d$.
%Since
%$$\exp\left\{-\varepsilon \Big(\frac{\ln^\sigma (1+ |\xi|)}{ W(\ln(1+|\xi|))}\Big)^{\frac{1}{\sigma-1}}%\right\}
%=(1+|\xi|)^{-\varepsilon \Big(\frac{\ln (1+ |\xi|)}{ W(\ln(1+|\xi|))}\Big)^{\frac{1}{\sigma-1}} },
%\quad \xi \in \mathbb R^d,$$
%from \eqref{NejednakostPaley} it follows that
%$\xi^{\alpha}\hat u(\xi)$ is integrable for all $\alpha\in \mathbb N^d$.

For arbitrary $\tau>0$, we take $A_{\sigma}$ be as in \eqref{KonacnaocenaAsocirana}, and
set $h= 2 A_{\sigma} \tau^{-\frac{1}{\sigma-1}}$ in \eqref{NejednakostPaley}. Then the Fourier inversion formula, together with \eqref{KonacnaocenaAsocirana} and \eqref{NejednakostPaley}, implies

\begin{multline*}
    |D^{\alpha} u(x)| = \left|\int_{\mathbb R^d}\xi^{\alpha}\widehat u(\xi) e^{2\pi i x\xi}d\xi\right|  \\
  \leq C \int_{\mathbb R^d}|\xi|^{|\alpha|} \exp\left\{-2 A_{\sigma} \tau^{-\frac{1}{\sigma-1}} \Big(\frac{\ln^\sigma (1+ |\xi|)}{ W (\ln(1+|\xi|))}\Big)^{\frac{1}{\sigma-1}}\right\} d\xi\\
  \leq  C_1 \sup_{\xi \in \mathbb R^d} \Big(|\xi|^{|\alpha|}\exp\{-T_{\tau,\sigma}(1+|\xi|)\}\Big)\int_{\mathbb R^d} \exp\left\{-A_{\sigma}  \Big(\frac{\ln^\sigma (1+ |\xi|)}{  \tau W (\ln(1+|\xi|))}\Big)^{\frac{1}{\sigma-1}}\right\}d\xi\\
   \leq C_2  \sup_{\xi \in \mathbb R^d}\frac{|\xi|^{|\alpha|}}{\exp\{T_{\tau,\sigma}(|\xi|)\}},\quad \xi \in \mathbb R^d,
\end{multline*} for suitable $C_2>0$. Then \eqref{strongerestimate} follows from
\eqref{asocirana}.
\end{proof}

%%
%\subsection{Matrix approach}
%%

%%%%%
\section{Wave-front sets for extended Gevrey regularity} \label{sec:WFs}
%%%%%

%%
\subsection{Wave-front set and singular support}

Wave-front sets measure different types of directional singularities. For example,
\begin{equation}
\label{wavefrontpodskup1}
\WF(u)\subsetneq \WF_t(u) \subsetneq \WF_A(u)\,, \;\;\; t>1,
\end{equation}
where $u \in \mathcal D'(\mathbb R^d)$, $\WF$ is the classical $(C^{\infty})$ wave-front set, $\WF_t$ is the Gevrey wave-front set, and $\WF_A$ is analytic wave-front set, we refer to \cite{Foland, H, Rodino} for precise definitions.

In this section we introduce wave-front sets which detect
singularities that are "stronger" then the classical $C^{\infty}$
singularities and "weaker" than any Gevrey singularity. Moreover,
the usual properties (such as pseudo-local property), which hold for
wave-front sets quoted in \eqref{wavefrontpodskup1}, are preserved when considering the new type of singularities.

For simplicity, here we consider wave-front sets ${\WF}_{\{\tau,\sigma\}}(u)$ in terms of extended Gevrey regularity of Roumieu type. Results on ${\WF}_{(\tau,\sigma)}(u)$ of Beurling type are analogous, cf. \cite[Remark 3.2]{PTT-02}.

\par

\begin{Definition}
\label{Wf_t_s1}
Let $u\in \mathcal D'(\mathbb R^d)$, $\tau>0$, $\sigma>1$, and $(x_0,\xi_0)\in \mathbb R^d \times\mathbb R^d\backslash\{0\}$. Then $(x_0,\xi_0)\not \in {\WF}_{\{\tau,\sigma\}}(u)$ if there exists
a conic neighborhood $\Gamma_0$ of $\xi_0$, a compact set $ K\subset\subset \mathbb R^d $, and
$\phi\in \mathcal D_{\{\tau,\sigma\}} (\mathbb R^d) $, $\supp \phi = K$,  $\phi (x_0)  \neq 0$, and such that
\begin{equation}
\label{WFset-definition}
|\widehat{\phi u}(\xi)|\leq C \frac{h^{N^{\sigma}} N^{\tau N^{\sigma}}}{|\xi|^N},\quad N\in {\mathbb N}\,,\xi\in \Gamma_0\,,
\end{equation}
for some $h>0$ and $C>0$.
\end{Definition}

Definition \ref{Wf_t_s1} does not depend on the choice of the cut-off function
$\phi\in \mathcal D_{\{\tau,\sigma\}} (\mathbb R^d) $ with given properties. We refer to
\cite[Theorem 4.2]{PTT-03} for the proof of such independence, and note that the inverse closedness property of $\mathcal E_{\{\tau,\sigma\}} (\mathbb R^d)$ (Theorem \ref{teoremaInverseClosedness}) is used in the proof. Thus, $(x_0,\xi_0)\not \in {\WF}_{\{\tau,\sigma\}}(u)$ if  \eqref{WFset-definition} holds {\em
for all} $\phi\in \mathcal D_{\{\tau,\sigma\}} (\mathbb R^d) $, $\supp \phi = K$,  $\phi (x_0)  \neq 0$, and sometimes it is convenient to assume that  $\phi (x_0)  \equiv 1$ in a neighboorhood of $x_0 \in  \mathbb R^d $,
cf. \cite{Rodino}.

Let $u\in \mathcal D'(\mathbb R^d)$. Then ${\WF}_{\{\tau,\sigma\}}(u)$
is a closed subset of $\mathbb R^d \times\mathbb R^d\backslash\{0\}$, and for every $\tau>0$ and $\sigma>1$ we have
\begin{equation*}
    \WF(u)\subsetneq {\WF}_{\{\tau,\sigma\}}(u)\subsetneq \WF_{t}(u) \subsetneq{\WF}_A (u).
\end{equation*}

The singular support of a  distribution  $u\in \mathcal D'(\mathbb R^d)$  with respect to  extended Gevrey regularity is the complement of the set of points in which $u$ locally belongs to
$\mathcal E_{\{\tau,\sigma\}} (\mathbb R^d)$:

\par

\begin{Definition}
\label{DefinicijaSingSup}
Let $\tau>0$, $\sigma>1$, and $u\in \mathcal  D'(\mathbb R^d)$.
Then $x_0\not \in \sing_{\{\tau,\sigma\}}(u)$ if and only if
there exists a neighborhood $\Omega$ of $x_0$ such that $u\in \mathcal E_{\{\tau,\sigma\}}(\Omega)$.
\end{Definition}

Here, $u\in \mathcal E_{\{\tau,\sigma\}}(\Omega)$ means that $u$ satisfies the conditions of  Definition \ref{def:extendedGervey}, i.e. \eqref{NewClassesInd},
%or \eqref{NewClassesProj}
with $\mathbb R^d$ replaced by its open subset $\Omega$ at each occurrence.

The next result is a consequence of Definition \ref{Wf_t_s1} and \ref{DefinicijaSingSup}, we refer to \cite{PTT-02} for the proof.

\begin{Theorem}
\label{SingularSupportTheorema}
Let $\tau>0$, $\sigma>1$, $u\in \mathcal D'(\mathbb R^d)$, and let
$\pi_1:\mathbb R^d \times \mathbb R^d\backslash \{0\}\to \mathbb R^d$ be the standard projection given by $\pi_1(x,\xi)=x$. Then
\begin{equation*}%\label{supportInkluzije}
     \sing_{\{\tau,\sigma\}}(u)=\pi_1(\WF_{\{\tau,\sigma\}}(u)).
\end{equation*}
\end{Theorem}

\subsection{Characterization of wave-front sets via the STFT}
%%

%%
%\subsection{Associated function with two parameters}
%%

For the estimates of the short-time Fourier transform it is convenient to consider the following refinement of the associated function  $T_{\tau,\sigma}$ (see   \eqref{asocirana}).

The \emph{two-parameter associated function}  $T_{\tau,\sigma}(h,k)$
to the sequence $M^{\tau,\sigma}_p=p^{\tau p^\sigma}$, $ p \in \mathbb{N}$,  $\tau>0$, $\sigma>1$ is given by
\begin{equation} \label{asociranaProduzena}
T_{\tau,\sigma}(h,k)= \sup_{p\in \mathbb N} \ln_+\frac{h^{p^{\sigma}}k^{p}}{M_p^{\tau,\sigma}}, \;\;\; k>0.
\end{equation}
When $h=1$ we recover  the associated function with the sequence $(M^{\tau,\sigma}_p)$, i.e.
\begin{equation*}
    T_{\tau,\sigma}(k) = T_{\tau,\sigma} (1,k), \quad k>0.
\end{equation*}

Sharp asymptotic estimates for $ T_{\tau,\sigma}(h,k)$ in terms
of the principal branch of the Lambert function are given in \cite{PTT-03}.

%where it is proved that for some $A_1,A_2>0$ and $B_1, B_2\in \mathbb R$
%(depending on $\tau\sigma$ and $h$) the following estimates hold:
%\begin{multline}
%\label{nejednakostzaTeoremu1}
%A_1  {W^{-\frac{1}{\sigma-1}}({{\mathfrak R}(h,k)})}\,{\ln_+}^{\frac{\sigma}{\sigma-1}}k +B_1 \leq
%T_{\tau,\sigma}(h,k)\\ \leq
% A_2 {W^{-\frac{1}{\sigma-1}}({{\mathfrak R}(h,k)})}\,{\ln_+}^{\frac{\sigma}{\sigma-1}}k +B_2,
%\end{multline}
%where
%$$
%{\mathfrak R}(h,k):=h^{-\frac{\sigma-1}{\tau}}e^{\frac{\sigma-1}{\sigma}}\frac{\sigma-1}{\tau %\sigma}\ln(e+k),\quad h,k>0,
%$$
%and $W$ is the principal branch of the Lambert function.
%\end{remark}

Here we recall \cite[Lemma 2]{TT0}, a simple result which relates $T_{\tau,\sigma}$ and $T_{\tau,\sigma}(h,\cdot)$, $h>0$.

\begin{Lemma}\label{OsobineAsocirane}
Let $T_{\tau,\sigma}(h,k)$ be given by \eqref{asociranaProduzena}, and let $T_{\tau,\sigma}(k)$
be given by \eqref{asocirana}.
Then for any given  $h>0$ and $\tau_2>\tau>\tau_1>0$ there exists $A,B\in {\mathbb R}$ such that
\begin{equation*}
    T_{\tau_2,\sigma}(k)+A \leq T_{\tau,\sigma}(h,k)\leq  T_{\tau_1,\sigma}(k)+B,\quad k>0.
\end{equation*}
\end{Lemma}

It is known that the classical wave-front set $ \WF(u) $ can be described by the means of the short-time Fourier transform, see \cite{PrangoskiPilip}. Related characterization of
${\WF}_{\{\tau,\sigma\}}(u)$ is given in \cite{TT-stft}. Here we provide a slightly different statement, and a more detailed proof.

Let there be given $f, g \in L^2 (\mathbb{R}^d).$ The short-time Fourier transform (STFT) of
$ f  $ with respect to the window $g$  is given by
\begin{equation*} \label{STFTdef}
V_g f (x,\xi) = \int e^{-2\pi i t \xi} f(t)  \overline{ g(t-x)}dt
= \langle f, M_\xi T_x g\rangle, \;\;\; x,\xi \in \mathbb{R}^d.
\end{equation*}
We observe that the definition of $V_g f $ makes sense when $f$ and $g$ belong to any pair of dual spaces, extending the inner product in $L^2 (\mathbb{R}^d)$ as it is mentioned in section \ref{notation}.

\par

%%%%%%%

We first observe that if $(x_0,\xi_0)\not \in {\WF}_{\{\tau,\sigma\}}(u)$ then
%\eqref{WFset-definition} implies
\begin{equation}
\label{WFsetinfimum}
|\widehat{\phi u}(\xi)|\leq C \inf_{N\in {\mathbb N}}
\frac{h^{N^{\sigma}} N^{\tau N^{\sigma}}}{|\xi|^N},\qquad \xi\in \Gamma_0,
\end{equation}
for some $h>0$, $C>0$, and $\phi$ satisfying the conditions of Definition  \ref{Wf_t_s1}.
By \eqref{asociranaProduzena} (and Lemma \ref{lema:katz-komatsu}) it follows that \eqref{WFsetinfimum} is equivalent to
\begin{equation}
\label{WFsetbyProduzena}
|\widehat{\phi u}(\xi)|\leq C e^{- T_{\tau,\sigma} ( \frac{1}{h},|\xi|)},\quad \xi\in \Gamma_0.
\end{equation}

Next  we resolve $\WF _{\{\tau,\sigma \} } (u)$ of $u\in \mathcal D'(\mathbb R^d)$  by considering the
asymptotic behavior of its STFT.

In the sequel $ \phi \in \mathcal D_{\{\tau,\sigma\}}  ^K (\mathbb R^d) $ means that
$ \phi \in \mathcal D_{\{\tau,\sigma\}}  (\mathbb R^d) $ and $\supp \phi  = K$.

\begin{Theorem}
\label{NezavisnostWFThm}
Let $u\in \mathcal D'(\mathbb R^d)$, and $\tau>0$, $\sigma>1$.
Then $(x_0,\xi_0)\not\in \WF_{\{\tau,\sigma\}}(u)$  if and only if
there exists a conic neighborhood  $\Gamma_0$ of $\xi_0$, a compact neighborhood
$ K $ of $x_0$, and
\begin{equation}
\label{cut-off-g}
g\in \mathcal D_{\{\tau,\sigma\}} ^{K_0} (\mathbb R^d), \qquad
K_{x_0}=\{y\in \mathbb R^d\,|\, y+x_0\in K\},
\qquad  g (0) \neq 0,
\end{equation}
such that
\begin{equation}
\label{TeoremaWFUslov2}
|V_{g} u(x,\xi)|\leq C e^{- T_{\tau,\sigma} ( k ,|\xi|) },\quad  x\in K,\,\xi\in \Gamma_0,
\end{equation}
for some $k>0$, and $ C>0$.
\end{Theorem}

\begin{proof} We follow the idea presented in \cite{PrangoskiPilip}, and give the proof to enlighten the difference between $\WF(u)$ and $ {\WF}_{\{\tau,\sigma\}}(u)$.

($\Rightarrow$) Assume that there is a conic neighborhood  $\Gamma$ of $\xi_0$, a
compact set $ K_1 $ in $ \mathbb R^d $, so that
for any $\phi\in \mathcal D_{\{\tau,\sigma\}} ^{K_1} (\mathbb R^d) $,
such that $\phi (x_0) \neq 0$, the estimate \eqref{WFsetbyProduzena} holds for some $C,h>0$.
Without loss of generality, we may assume that $K_1 =\overline{B_r (x_0)}$ for some $r>0$.

By \eqref{WFsetbyProduzena} it follows that the set
\begin{equation*}
    H_h = \{ e^{T_{\tau,\sigma} ( \frac{1}{h},|\xi|)} e^{-i\xi \cdot}  u(\cdot) \quad | \quad
\xi \in \Gamma_0 \}
\end{equation*}
is weakly bounded, and weakly continuous (since $D_{\{\tau,\sigma\}} ^{K_1} (\mathbb R^d) $ is barelled).

\par

Let  $K =\overline{B_{r/2} (x_0)}$, and consider the window
$ g \in \mathcal D_{\{\tau,\sigma\}} ^{ K_{x_0} } (\mathbb R^d)$, such that
$g  \neq 0$ on a neighborhood of $0$.

Then $ \phi \equiv T_{x} \bar{g}  \in \mathcal D_{\{\tau,\sigma\}} ^{K_1} (\mathbb R^d) $,
and $ \phi \neq 0$ on a neighborhood of $x_0$. By the equicontinuity of $H_h$
it follows that
\begin{multline}
\label{equicont}
| \langle  e^{T_{\tau,\sigma} ( \frac{1}{h},|\xi|)} e^{-i\xi \cdot}  u(\cdot),  T_{x} \bar{g}(\cdot)
\rangle | \leq
C_1 \sup_{|\alpha| \leq N} \sup_{t \in K_1} | D^{\alpha} g(t-x)|
\\
= C_1 \sup_{|\alpha| \leq N} \| D^{\alpha} g\|_{L^\infty}
\leq C
\end{multline}

>From the definition of STFT it follows that
\begin{equation*}
    V_{g}u(x,\xi) = \widehat{u T_x \bar{g} }(x,\xi),\;\;\; x,\xi \in \mathbb{R}^d.
\end{equation*}
This, together with
\eqref{equicont} implies
\begin{equation*}
    |V_{g}u(x,\xi)|=|\widehat{u T_x \bar{g} }|
\leq
%A \inf_{N\in \mathbb N}\frac{h^{N^{\sigma}}N^{\tau N^{\sigma}}}{|\xi|^N}=
C e^{-T_{\tau,\sigma}(1/h,|\xi|)},
\qquad \xi\in \Gamma,
\end{equation*}
for all $x\in K$, and for some constants $C,h>0$, which gives \eqref{TeoremaWFUslov2}.

Notice that we actually proved that \eqref{TeoremaWFUslov2} holds {\em for any} $g$ satisfying
\eqref{cut-off-g}.

\par

($\Leftarrow$) Let the window $g \in \mathcal D_{\{\tau,\sigma\}} ^{ K_{x_0}} (\mathbb R^d)$,
$g \neq 0$ in a neighborhood of $0$. Then
$\psi=T_{x_0}\bar{g}\in \mathcal D_{\{\tau,\sigma\}} ^K  (\mathbb R^d)$, $\psi (x_0) \neq 0 $,
and
\begin{equation*}
    |\widehat{\psi u}(\xi)|=|V_{g} u (x_0,\xi)|
\leq  A  e^{-T_{\tau,\sigma}(k,|\xi|)}
\leq C \frac{h^{N^{\sigma}}N^{\tau N^{\sigma}}}{|\xi|^N},\quad N\in \mathbb N, \quad
\xi \in \Gamma_0,
\end{equation*}
for some $C>0$ and $h= 1/k > 0$, and the proof is complete.
\end{proof}

By using Lemma \ref{OsobineAsocirane} and Theorem \ref{TeoremaAsocirana}
we can express the decay estimate  \eqref{TeoremaWFUslov2} in terms of the Lambert function as follows.

\begin{Corollary} \label{posledica1}
Let $u\in \mathcal D'(\mathbb R^d)$, $\tau>0$, $\sigma>1$, and let $W$ denote the Lambert function.
If $(x_0,\xi_0)\not\in \WF_{\{\tau,\sigma\}}(u)$ then
there exists a conic neighborhood  $\Gamma_0$ of $\xi_0$, a compact neighborhood
$ K $ of $x_0$, and $g$ satisfying  \eqref{cut-off-g} such that
\begin{equation*}
   % \label{TeoremaWFUslov3}
|V_{g} u(x,\xi)|\leq C e^{- c \left ( \frac{\ln^{\sigma }(|\xi|) }{\tau _2 W (\ln (|\xi|))} \right )^{\frac{1}{\sigma -1} } },\quad  x\in K,\,\xi\in \Gamma_0,
\end{equation*}
for some $c>0$, $ C>0$, and any $\tau_2 > \tau$.

Conversely, if there exists a conic neighborhood  $\Gamma_0$ of $\xi_0$, a compact neighborhood
$ K $ of $x_0$, and $g$ satisfying  \eqref{cut-off-g} such that
\begin{equation}
   \label{TeoremaWFUslov4}
|V_{g} u(x,\xi)|\leq C e^{- c \left ( \frac{\ln^{\sigma }(|\xi|) }{\tau _1 W (\ln (|\xi|))} \right )^{\frac{1}{\sigma -1} } },\quad  x\in K,\,\xi\in \Gamma_0,
\end{equation}
holds for some $c>0$, $ C>0$, and $\tau_1 < \tau$,
then $(x_0,\xi_0)\not\in \WF_{\{\tau,\sigma\}}(u)$.
\end{Corollary}

As a combination of results from Theorem \ref{Thm:Peli-Viner} ii) and Theorem \ref{SingularSupportTheorema}, we can use Corollary \ref{posledica1} to characterize local regularity of  $u\in \mathcal D'(\mathbb R^d)$. Namely if \eqref{TeoremaWFUslov4} holds, then
$x_0 \not\in \sing_{\{\tau,\sigma\}}(u)$, so that
$u\in \mathcal E_{\{\tau,\sigma\}}(\Omega)$ in a  neighborhood $\Omega$ of $x_0$
(see Definition \ref{DefinicijaSingSup}).

\subsection{Propagation of singularities}
One of the main properties of wave-front sets is microlocal hypoelipticity. We first  recall the notion of the characteristic set of an operator and the main property of its principal symbol.

If $ P(x,D)=\sum_{|\alpha|\leq m} a_{\alpha } (x) D^{\alpha}$ is a differential operator of order $m$ in $\mathbb R^d$ and
$a_\alpha \in C^\infty (\mathbb R^d)$, $|\alpha|\leq m $,  then
its characteristic set is given by
\begin{equation*}
    {\rm Char}(P(x,D))=
\bigcup_{x\in \mathbb R^d}
\left\{(x,\xi)\in \mathbb R^d\times \mathbb R^d\backslash\{0\} \mid P_m(x,\xi)=0 \right\}.
\end{equation*}

Here $P_m(x,\xi)=\sum_{|\alpha|=m}a_{\alpha}(x)\xi^{\alpha} \in C^{\infty}(\mathbb R^d\times\mathbb R^d\backslash\{0\})$
is the principal symbol of $ P(x,D).$
If ${\rm Char}(P(x,D))= \emptyset$, then the operator $P(x,D)$ is {\em hypoelliptic}.

Noe, for the Roumieu wave-front $\WF_{\{\tau,\sigma\}}(u)$ we have the following theorem on the paopagation of singularities.

\begin{Theorem}
\label{FundamentalTheoremNK}
Let $\tau>0,$ $\sigma\geq 1$, $u\in \mathcal D'(\mathbb R^d)$ and let $P(x,D)=\sum_{|\alpha|\leq m}a_{\alpha}(x)D^{\alpha}$ be partial differential operator of order $m$ such that $a_{\alpha}(x)\in \mathcal E_{\{\tau,\sigma\}}(\mathbb R^d)$, $|\alpha|\leq m $. Then
\begin{equation*}
\label{FundamentalEstimateNK}
\WF_{\{2^{\sigma-1}\tau,\sigma\}}(f)\subseteq \WF_{\{2^{\sigma-1}\tau,\sigma\}}(u)\subseteq \WF_{\{\tau,\sigma\}}(f) \cup {\rm Char}(P(x,D)),
\end{equation*}
where $P(x,D)u=f$ in $\mathcal D'(\mathbb R^d)$.
In particular,
\begin{equation}
\label{hypo}
\WF_{0,\sigma}(f)\subseteq \WF_{0,\sigma}(u)\subseteq \WF_{0,\sigma}(f) \cup {\rm Char}(P(x,D)),
\end{equation}
where $ \WF_{0,\sigma}(u)= \bigcap_{\tau>0}\WF_{\{\tau,\sigma\}}(u)$.
\end{Theorem}

The proof of Theorem \ref{FundamentalTheoremNK} uses inverse closedness (Theorem
\ref{teoremaInverseClosedness}), Paley-Wiener type estimates (Theorem
\ref{Thm:Peli-Viner}), and contains nontrivial modifications of
the proof of \cite[Theorem 8.6.1]{H}. We refer to \cite{PTT-02} for a detailed proof. Note that if $\sigma=1$ and $\tau>1$, we recover the result for propagation of singularities when the coefficients are Gevrey regular functions, and  $\WF_{0,\sigma}(f) = \WF_{0,\sigma}(u)$ in \eqref{hypo} reveals the hypoellipticity of $(P(x,D)$.

%%%%%
\section{Applications}  \label{sec:appl}
%%%%%

\subsection{A strictly hyperbolic partial-differential equation}

Cicognani and Lorenz in \cite{CL}
considered the Cauchy problem for strictly hyperbolic $m$-th order
partial-differential equations (PDEs)  of the form
\begin{multline} \label{cauchy}
D_t ^m u = \sum_{j=0} ^{m-1} A_{m-j} (t,x, D_x) D^j _t u + f(t,x), \\
D^{k-1} _t u (0,x) = g_k (x), \quad (t,x) \in [0,T] \times  \mathbb R^d, \quad
k=1,\dots, m,
\end{multline}
where
\begin{equation*}
    A_{m-j} (t,x, D_x) =
\sum_{|\gamma| + j \leq m} a_{m-j,\gamma} (t,x)  D^\gamma _x,
\end{equation*}
where $f$ and $g_k$, $k=1,\dots, m$, satisfy certain Sobolev type regularity conditions (cf. (SH3-W) and (SH4-W) in \cite{CL}).

\noindent and studied well-posedness when the coefficients are low-regular in time, and smooth in space. More precisely, it is assumed that the coefficients $ a_{m-j,\gamma} $ satisfy conditions of the form
\begin{equation} \label{coeff-in-PDE}
|  D^\beta _x  a_{m-j,\gamma} (t,x) - D^\beta _x  a_{m-j,\gamma} (s,x)| \leq C K_{|\beta|}
\mu (|t-s|), \quad 0 \leq |t-s| \leq 1, \quad x \in \mathbb R^d,
\end{equation}
where $\mu$ is a modulus of continuity, and
$(K_{|\beta|})$ is a defining sequence (also called a weight sequence).

The modulus of continuity $\mu$ is used to describe the (low) regularity in time, whereas  $(K_{|\beta|})$ describes the regularity in space.

When $\mu $ is a weak modulus of continuity,
\begin{equation*}
    \mu (s) = s (\log  \frac{1}{s} + 1) \log ^{[m]} ( \frac{1}{s} ), \qquad s > 1,
\end{equation*}
(Log-Log$^{[m]}$-Lip-continuity), a suitable weight function $ \eta $ which defines the solution space is chosen to be
\begin{equation*}
    \eta (s) = \log (s)(\log^{[m]} (s))^{1+\varepsilon} + c_m,
\end{equation*}
where $\varepsilon>0$ is arbitrarily small and $c_m>0$ such that $\eta (s) \geq 1$ for all $s > 1$.

We refer to  \cite{CL} for a detailed analysis of  \eqref{cauchy}, and note that the relation between the modulus of continuity $\mu$ and the weight function $\eta$ is given by
\begin{equation*}
    \lim_{|\xi| \rightarrow \infty}
\frac{\mu (\frac{1}{\langle \xi \rangle}) \langle \xi \rangle}{\eta (\langle \xi \rangle)}
 = 0, \qquad\langle \xi \rangle ^2 = 1 + |\xi |^2, \;\; \xi \in  \mathbb R^d,
\end{equation*}
while the condition which links the weight sequence
%, i.e. the regularity of coefficients with respect to the sequence
$(K_{p})$ to the weight function
$\eta$  is given by
\begin{equation*}
    \inf_{p\in \mathbb N} \frac{K_p}{\langle \xi \rangle ^p} \leq C e^{-h \eta (\langle \xi \rangle)},
\end{equation*}
for some $h,C >0$, which is essentially the relation between  the Carleman associated function and
the Komatsu associated function as given in Lemma \ref{lema:katz-komatsu}.

One of the conclusions in \cite{CL} is that the Cauchy problem \eqref{cauchy} is well-posed  if $a_{m-j,\gamma} (t,x) \in
 {\mathcal E}_{\{ 1, 2\}}( \mathbb{R}^d)$
uniformly in $x$ for every fixed $t$. In other words, the sequence $(K_{|\beta|})$ in
\eqref{coeff-in-PDE} is given by $\displaystyle K_p = p^{p^2}$.

%%%
\subsection{Generalized definition of ultradifferentiable classes}
\label{matrices-approach}
%%%

It is recently demonstrated in \cite{Garrido-2} that the extended Gevrey classes
are prominent example of ultradifferentiable functions defined in the framework of
generalized weighted matrices approach.

The main idea behind the weighted matrices approach as given in \cite{Schindl} and \cite{RS}
is to establish a general framework for considering the  Braun-Meise-Taylor and Komatsu approach
to ultradifferentiable functions in a unified way. To include the extended Gevrey classes which are called PTT-classes in \cite{Garrido-1}
and \cite{Garrido-2} (after Pilipovi\'c-Teofanov-Tomi\'c), the so-called exponential sequences $ \Phi = (\Phi_p)_{p \in  \mathbb N}$, and the related generalized weighted matrix setting
are introduced in  \cite{Garrido-2}.
One of the main observations in  \cite{Garrido-2} is that the exponential sequences $ \Phi$ (such as $(h^{p^\sigma})_{p \in  \mathbb N}$, for some $h>0$) yield
"ultradifferentiable classes beyond geometric growth factors", under mild regularity and growth assumptions on $ \Phi$. In such context, PTT-classes constitute a genuine examples of
of ultradifferentiable functions defined by weight matrices.

This approach reveals that, apart from stability properties mentioned in Section \ref{SectionSpaces}, PTT-classes enjoy almost analytic extension
\cite{Rainer}, and almost harmonic extension
\cite{DebVin}. Moreover, PTT-classes are  a convenient tool for the study of Borel mappings. More precisely, the asymptotic Borel mapping,
which sends a function into its series of asymptotic expansion in a sector, is known to be surjective for arbitrary openings in the framework of ultraholomorphic classes associated with sequences of rapid growth.
By using the PTT-classes ${\mathcal E}_{\{\tau, \sigma\}}( \mathbb{R}^d)$, given by $M^{\tau,\sigma}_p=p^{\tau p^\sigma}$, $ p \in \mathbb{N} $, $\tau>0$,  $ 1 < \sigma < 2$,
Jim\'enez-Garrido, Lastra and Sanz, presented a constructive proof of
the surjectivity of the Borel map in sectors of the complex plane for the ultraholomorphic class associated with those specific sequences.
In fact, the asymptotic behavior of the associated function given in terms of the Lambert function
(see Theorem \ref{TeoremaAsocirana}) plays a prominent role in these investigations.
We refer to  \cite{Garrido-1} for more details.

%%%%%%%%%%%%%%%%%%%%%%%%%%%%%%%%%%%%%%%%%%
%\section{Discussion}
%
%The family of extended Gevrey classes introduced in \cite{PTT-01, PTT-02}
%can be introduced within generalized weighted matrices approach to ultradifferentiable functions,  \cite{Garrido-2}. Due to the properties of their defining sequences $(M_p ^{\tau,\sigma})$, $\tau  > 0$, $\sigma>1$, the asymptotic behavior of the associated function $T _{\tau ,\sigma} $ can be conveniently described via the Lambert function $W$. This is an important feature in applications, cf. \cite{CL,Garrido-1}. We believe that these insights will be useful in future investigations, both from theoretical and applied point of view.

%%%%%%%%%
\section*{Appendix}
%%%%%%%%%

\begin{proof} (proof of Lemma \ref{osobineM_p_s})
$(M.1)$ obviously holds for $p =1$. For $p-1 \in \mathbb N$ we observe that
$\ln x^{\tau x^\sigma } = \tau x^\sigma \ln x$ is a convex function when $x>1$,
which implies
\begin{equation*}
    2\tau p^\sigma \ln p \leq
\tau (p-1)^\sigma \ln (p-1) + \tau (p+1)^\sigma \ln (p +1),
\end{equation*}
and $(M.1)$ follows after taking exponential.

To show $\widetilde{(M.2)}$ we use  $(p+q)^{\sigma}\leq 2^{{\sigma}-1}(p^{\sigma}+q^{\sigma})$
which implies
\begin{equation*}
    (p+q)^{\tau (p+q)^{\sigma}}\leq (p+q)^{\tau 2^{\sigma-1}p^{\sigma}}(p+q)^{\tau 2^{\sigma-1}q^{\sigma}},
\quad p,q \in \mathbb N.
\end{equation*}
The logarithm of the first term on the right hand side of the inequality can be estimated as follows:
\begin{eqnarray*}
 \tau 2^{\sigma-1}p^{\sigma}\ln(p+q)&=& \tau 2^{\sigma-1}p^{\sigma}\Big(\ln p +\ln\Big(1+\frac{q}{p}\Big)\Big) \\
& \leq &
 \tau 2^{\sigma-1}p^{\sigma}\ln p + \tau 2^{\sigma-1}qp^{{\sigma}-1} \\
 & \leq & \tau 2^{\sigma-1}p^{\sigma}\ln p + \tau 2^{\sigma-1}(p+q)^{\sigma}.
\end{eqnarray*}
By taking exponential we obtain
\begin{equation*}
    (p+q)^{ \tau 2^{\sigma-1}p^{\sigma}}\leq p^{\tau 2^{\sigma-1}p^{\sigma}}e^{\tau 2^{\sigma-1}(p+q)^{\sigma}},
\end{equation*}
and by  replacing the roles of $p$ and $q$ we get
\begin{equation*}
    (p+q)^{ \tau 2^{\sigma-1}q^{\sigma}}\leq q^{\tau 2^{\sigma-1}q^{\sigma}}e^{\tau 2^{\sigma-1}(p+q)^{\sigma}},
\end{equation*}
thus
\begin{equation*}
    (p+q)^{\tau (p+q)^{\sigma}}\leq p^{\tau 2^{\sigma-1}p^{\sigma}}  q^{\tau 2^{\sigma-1}q^{\sigma}} e^{\tau 2^{\sigma}(p+q)^{\sigma}},
\end{equation*}
and $\widetilde{(M.2)}$ is proved.

Let us show that  $\widetilde{(M.2)'}$ holds true.
Put ${\sigma}=n+ \delta$ where $n\in \mathbb N$, $0<\delta\leq 1$. If
${\sigma}\not\in \mathbb N$ then $n=\lfloor {\sigma}\rfloor,$ $0<\delta<1,$ while
$n= {\sigma} -1,\, \delta=1,$ if $ {\sigma}\in \mathbb N$.
By the binomial formula we have:
\begin{eqnarray*}
( p+1)^{\sigma} &\leq& (p+1)^n(p^{\delta}+1) \\
&=& p^{\sigma}+\sum_{k=1}^{n}\binom{n}{k}p^{{\sigma}-k}
+ \sum_{k=0}^{n}\binom{n}{k} p^{n-k} \\
& =  & p^{\sigma} + 2^{n} p^{{\sigma}-1} + 2^n p^n  \\
& \leq  & p^{\sigma} + 2^{n+1} p^{{\sigma}- \delta},
\end{eqnarray*}
wherefrom
\begin{equation} \label{m.2}
\tau (p+1)^{\sigma}\ln(1+p) \leq \tau  p^{\sigma}\ln(1+p)+ \tau 2^{n+1}q^{\sigma}p^{{\sigma}-\delta}\ln(1+p).
\end{equation}
%We will use
%the fact that for any $ \alpha > 0 $ there exists $A>0$ such that $ \ln x \leq A x^\alpha ,$ $x\geq 1$.
%Therefore $p\leq C^{p^{\delta}}$, for some $C>1$ when $ p\in \Z$ and $0<\delta\leq 1$.
The first term on the right hand side of the inequality \eqref{m.2} can be estimated by
\begin{eqnarray*}
\label{m.2one}
\tau p^{\sigma}\ln(1+p)&= & \tau p^{\sigma} \ln p (1+\frac{1}{p})  =
\tau p^{\sigma} (\ln p  + \ln (1+\frac{1}{p}))\\
&\leq& \tau p^{\sigma}\ln p + \tau p^{\sigma - 1} \leq \tau p^{\sigma}\ln p + \tau p^{\sigma},
\end{eqnarray*}
while for the second term  we use
\begin{eqnarray*} \label{m.2two}
\tau 2^{n+1}p^{{\sigma}-\delta}\ln(1+p)&=&
\tau 2^{n+1} p^{{\sigma}-\delta}
(\ln p +\ln (1+\frac{1}{p}))\nonumber\\
& \leq &
\tau 2^{n+1}p^{{\sigma}} C + \tau 2^{n+1}p^{{\sigma}}\ln 2\,.
\end{eqnarray*}
Here we used $ p^{-\delta } \ln p \leq  C$ for some $C>0$.
Thus we have
\begin{eqnarray*} \label{m.2three}
\tau (p+1)^{\sigma}\ln(1+p)  &\leq&
\tau p^{\sigma}\ln p +  \tau p^{\sigma} (1+ 2^{n+1} \tilde C),
\end{eqnarray*}
with $ \tilde C = C + \ln 2$. By taking exponential we obtain
\begin{equation*}
    (p+1)^{\tau (p+1)^{\sigma}}
\leq B^{p^\sigma} M^{\tau, \sigma} _p,
\end{equation*}
for some $B>0$, which gives $\widetilde{(M.2)'}$.

\par

To prove $(M.3)'$ we use $2 \leq  (1+1/p)^p$, $p\in \mathbb N$, which gives
\begin{equation*} %\label{nebitno}
{\tau}\, p^{{\sigma}-1} \ln 2 \leq \tau p^{{\sigma}} \ln \Big(1+\frac{1}{p}\Big)\leq \tau p^{{\sigma}-1},
\quad p\in \mathbb N,
\end{equation*}
i.e.
\begin{equation}   \label{expnejednakost}
2^{\tau p^{{\sigma}-1}}\leq \Big(1+\frac{1}{p}\Big)^{\tau p^{{\sigma}}}\leq e^{\tau p^{{\sigma}-1}},\,
\quad p\in \mathbb N.
\end{equation}
The  left hand side of \eqref{expnejednakost} and
\begin{equation*}
    p^{\sigma}\geq (p-1)^{\sigma-1}p=(p-1)^{\sigma}+(p-1)^{\sigma-1},  \qquad
p\in \mathbb N,
\end{equation*}
give
\begin{eqnarray*} %\label{redM.3'}
\sum _{p=1}^{\infty}\frac{(p-1)^{\tau (p-1)^{\sigma}}}{p^{\tau p^{\sigma}}}
&\leq&
\sum _{p=1}^{\infty}
\frac{(p-1)^{\tau (p-1)^{\sigma}}}{ p^{\tau ( (p-1)^{\sigma} + (p-1)^{{\sigma}-1}) }} \nonumber\\
&=&
\sum _{p=1}^{\infty}
\left ( (1 - \frac{1}{p})^{\tau (p-1)^{\sigma}} \right )
\frac{1}{p^{\tau (p-1)^{{\sigma}-1} }} \\ \nonumber
&\leq&
\sum _{p=1}^{\infty} \frac{1}{{(2p)^{\tau {(p-1)^{{\sigma}-1}}}}}   <\infty, \nonumber
\end{eqnarray*}
which is $(M.3)'$.
\end{proof}

\begin{proof} (proof of Lemma \ref{lm:konstante-i-nizovi})
$i)$ $(\Rightarrow) $ Let $a_p\leq C h^{p^{\sigma}}$, $p\in {\mathbb N}_0$, for some $C,h>0$, let $\left(r_j\right) $ be any sequence in $ \mathcal{R}$, and let $j_0 \in {\mathbb N}_0$ be such that $\displaystyle \frac{h}{r_j}\leq 1$, for all $j\geq j_0$. Then
\begin{equation*}
a_p\leq C h^{p^{\sigma}}=C \prod_{j=1}^{j_0}h \prod_{j=j_0+1}^{p^{\sigma}}r_j \frac{h}{r_j}\leq C h^{j_0}\prod_{j=1}^{p^{\sigma}}r_j
\leq C_1 \prod_{j=1}^{p^{\sigma}}r_j = C_1 R_{p,\sigma},
\end{equation*}
for large enough $p \in \mathbb{N}_0$ and suitable $C_1>0$. This proves \eqref{pre2}.

$(\Leftarrow) $ The opposite part we prove by contradiction. Assume that \eqref{pre2} holds for arbitrary $\left(r_j\right) \in  \mathcal{R}$, and that
\begin{equation*}
\displaystyle \sup \left\{\frac{a_{p}}{h^{p^{\sigma}}}: p \in \mathbb{N}_0\right\}=\infty
\qquad \text{ for every } \qquad
h>0.
\end{equation*}
%%%%%%%%%%%%%%%%%%%%%%%%%%%%%%%%%%%%%%%%%
Thus, for every $n\in\mathbb{N}$ and $h:=n$ there exists $p_n\in \mathbb{N}$ such that
\begin{equation*}
\frac{a_{p_n}}{n^{ \lfloor p_n^\sigma \rfloor}}>n.
\end{equation*}

If $n=1$, then there exists $p_1\in \mathbb{N}$ such that $a_{p_1}>1$, and obviously
\begin{equation*}
\frac{a_{p_1}}{r_1 r_2\dots r_{\lfloor p_1^\sigma \rfloor}}>1
\end{equation*}
if $r_1=r_2=\dots=r_{\lfloor p_1^\sigma \rfloor}=1$.

Similarly, when $n=2$, there exists $p_2>p_1$ such that  $ \lfloor p_2 ^\sigma \rfloor> \lfloor p_1 ^\sigma \rfloor$, and
\begin{equation*}
\frac{a_{p_2}}{2^{\lfloor p_2^\sigma \rfloor}}>2.
\end{equation*}
By choosing $r_{\lfloor p_1^\sigma \rfloor + 1}=r_{\lfloor p_1^\sigma \rfloor + 2}
=\dots=r_{\lfloor p_2^\sigma \rfloor}=2$ we get
$\prod_{j=1}^{\lfloor p_2^\sigma \rfloor}r_j=2^{\lfloor p_2^\sigma \rfloor - \lfloor p_1^\sigma \rfloor }$,
%Since $2^{\lfloor p_2^\sigma \lfloor}>\prod_{j=1}^{p_2^\sigma}r_j=2^{p_2^\sigma-p_1^\sigma}$ we obtain
wherefrom
\begin{equation*}
\frac{a_{p_2}}{r_1\dots r_{\lfloor p_2^\sigma \rfloor}} \geq \frac{a_{p_2}}{2^{\lfloor p_2^\sigma \rfloor}}>2.
\end{equation*}

Next, we take $p_3>p_2$ such that  $ \lfloor p_3 ^\sigma \rfloor> \lfloor p_2 ^\sigma \rfloor$, and
\begin{equation*}
\frac{a_{p_3}}{3^{\lfloor p_3^\sigma \rfloor}}>3,
\end{equation*}
so we can choose $r_{\lfloor p_2^\sigma \rfloor +1}=r_{\lfloor p_2^\sigma \rfloor+2}
=\dots=r_{\lfloor p_3^\sigma \rfloor}=3$ to obtain
\begin{equation*}
\prod_{j=1}^{\lfloor p_3^\sigma \rfloor }r_j
=
1^{\lfloor p_1^\sigma \rfloor}\cdot 2^{\lfloor p_2^\sigma \rfloor - \lfloor p_1^\sigma \rfloor}
\cdot
3^{\lfloor p_3^\sigma \rfloor - \lfloor p_2^\sigma \rfloor}
=\left(\frac{1}{2}\right)^{\lfloor p_1^\sigma \rfloor}\left(\frac{2}{3}\right)^{\lfloor p_2^\sigma \rfloor}3^{\lfloor p_3^\sigma \rfloor}<3^{\lfloor p_3^\sigma \rfloor}.
\end{equation*}
Thus for $n=3$ we get
\begin{equation*}
\frac{a_{p_3}}{r_1\dots r_{ \lfloor p_3^\sigma \rfloor}}=\frac{a_{p_3}}{\prod_{j=1}^{\lfloor p_3^\sigma \rfloor}r_j}>\frac{a_{p_3}}{3^{\lfloor p_3^\sigma \rfloor}}>3.
\end{equation*}

In the same fashion for any $n+1 \in \mathbb{N}$ we can find $p_{n+1}>p_n$ such that
 $\lfloor p_{n+1} ^\sigma \rfloor > \lfloor p_n ^\sigma \rfloor$, and by choosing
\begin{equation*}
r_{\lfloor p_{n} ^\sigma \rfloor +1}=r_{\lfloor p_n ^\sigma \rfloor+2}
=\dots=r_{\lfloor p_{n+1} ^\sigma \rfloor}= n+1
\end{equation*}
we obtain
\begin{equation*}
\frac{a_{p_{n+1}}}{\prod_{j=1}^{\lfloor p_n^\sigma \rfloor}r_j\cdot (n+1)^{\lfloor p_{n+1}^\sigma \rfloor - \lfloor p_n^\sigma \rfloor}}>\frac{a_{p_{n+1}}}{(n+1)^{\lfloor p_{n+1}^\sigma \rfloor}}>n+1.
\end{equation*}
By the construction it follows that $(r_j)\in  \mathcal{R}$, and for the sequence
\begin{equation*}
R_{p,\sigma}=\prod_{j=1}^{\lfloor p _n ^\sigma \rfloor} r_j
\end{equation*}
we obtain $\displaystyle\sup \left\{\frac{a_{p}}{R_{p,\sigma}}: p \in \mathbb{N}_0\right\}=\infty,$
which contradicts  \eqref{pre2}.

$ii)$ $(\Rightarrow) $ follows similarly as in $i)$.

$(\Leftarrow) $ Let \eqref{pre4} holds for every $h>0$, and put
\begin{equation*}
\displaystyle C_{h}:=\sup \left\{h^{p^{\sigma}} a_{p}:p\in N_0\right\}, \quad \text{ for } \quad h\geq 1.
\end{equation*}
We define
\begin{equation*}
H_0 = 1, \qquad H_j:=\sup \left\{\frac{h^j}{C_h}: h \geq 1\right\}, \quad j \in \mathbb{N}.
\end{equation*}
It is easy to see that $(H_j)$ is a well defined sequence which satisfies $(M.1)$, and that
$H_j/h^j$ tends to infinity as $j \rightarrow \infty$, for all $h\geq 1$. Therefore
$\left(r_j\right) \in  \mathcal{R}$, where $\displaystyle r_j=\frac{H_j}{H_{j-1}}$, $j\in {\mathbb N}$. We
note that
\begin{equation*}
H_{p^{\sigma}}a_p= \sup \left\{\frac{h^{p^{\sigma}}}{C_h}: h \geq 1\right\} a_p \leq 1,
\end{equation*}
and finally
\begin{equation*}
R_{p,\sigma} a_p = \Big(\prod_{j=1}^{p^{\sigma}}r_j\Big) a_p = H_{p^{\sigma}}a_p \leq 1,
\end{equation*}
which gives the statement.
\end{proof}

%%%%%%%%%%%%

%%%%%%%%%%%


\begin{thebibliography}{99999}
%%%%%%%%%%%%

\bibitem{Gevrey} Gevrey, M.
Sur la nature analitique des solutions des \'equations aux d\'eriv\'ees
partielle. Premier m\'emoire.
{\em Ann. Ec. Norm. Sup. Paris} {\bf 1918}, {\em 35}, 129--190.

\bibitem{Rodino} Rodino, L.
\textit{Linear Partial Differential Operators in Gevrey Spaces};
World Scientific, 1993.

\bibitem{ChenRodino} Chen, H.; Rodino, L.
General theory of PDE and Gevrey classes in General theory of partial differential equations and microlocal analysis.
{\em Pitman Res. Notes Math. Ser.} {\bf 1996}, {\em 349}, 6--81.

\bibitem{J} J\'ez\'equel, M.
Global trace formula for ultra-differentiable Anosov flows.
{\em Commun. Math. Phys.} {\bf 2021}, {\em 385}, 1771--1834.

\bibitem{CL} Cicognani, M.; Lorenz, D.
Strictly hyperbolic equations with  coefficients low-regular win time and smooth in space.
{\em J. Pseudo-Differ. Oper. Appl.} {\bf 2018}, {\em 9}, 643--675.

\bibitem{PTT-01} Pilipovi\'c, S.; Teofanov, N.;  Tomi\'c, F.
On a class of ultradifferentiable functions.
{\em Novi Sad J. Math.} {\bf 2015}, {\em 45 (1)}, 125--142.

\bibitem{PTT-02}  Pilipovi\'c, S.; Teofanov, N.;  Tomi\'c, F.
Beyond Gevrey regularity: Superposition and propagation of singularities.
{\em Filomat} {\bf 2018}, {\em 32 (8)}, 2763--2782.

\bibitem{H} H\"{o}rmander, L.
\textit{The Analysis of Linear Partial Differential Operators I};
Springer, 1990.

\bibitem{Komatsuultra1} Komatsu, H.
Ultradistributions, I: Structure theorems and a characterization.
{\em J. Fac. Sci. Univ. Tokyo, Sect. IA Math.} {\bf 1973}, {\em 20}, 25--105.

\bibitem{Garrido-2} Jim\'enez-Garrido, J.; Nenning, D.N.; Schindl, G.
On generalized definitions of ultradifferentiable classes.
{\em J. Math. Anal. Appl.} {\bf 2023}, {\em 526}, 127260.

\bibitem{Garrido-1} Jim\'enez-Garrido, J.; Lastra, A.; Sanz, J.
\textit{Extension Operators for Some Ultraholomorphic Classes Defined by Sequences of Rapid Growth};
Constr. Approx., 2023. https://doi.org/10.1007/s00365-023-09663-z

\bibitem{TTT-24} Teofanov, N.; Tomi\'c, F.; Tuti\'c, S.
Band-limited wavelets beyond Gevrey regularity.
{\em arXiv:2402.16426} {\bf 2024}

\bibitem{Siddiqi}  Siddiqi, J. A.
Inverse-closed Carleman algebras of infinitely differentiable functions. {\em Proc. Amer. Math. Soc.} {\bf 1990}, {\em 109}, 357--367.

\bibitem{P-13} Prangoski, B.
Laplace transform in spaces of ultradistributions.
{\em Filomat} {\bf 2013}, {\em 27 (5)}, 747--760.

\bibitem{LambF} Corless, R.M.; Gonnet, G.H.; Hare, D.E.G.; Jeffrey, D.J.; Knuth, D.E.
On the Lambert W function.
{\em Adv. Comput. Math.} {\bf 1996}, {\em 5}, 329--359.

\bibitem{Mezo} Mez\H{o}, I.
\textit{The Lambert W Function Its Generalizations and Applications};
CRC Press, Boca Raton, 2022.

\bibitem{PTT-03} Pilipovi\'c, S.; Teofanov, N.; Tomi\'c, F.
A Paley--Wiener theorem in extended Gevrey regularity.
{\em J. Pseudo-Differ. Oper. Appl.} {\bf 2020}, {\em 11}, 593--612.

\bibitem{Katznelson}  Katznelson, Y.
\textit{An Introduction to Harmonic Analysis};
John Wiley $\&$ Sons, Inc., 1968.

\bibitem{GelfandShilov} Gelfand, I. M.; Shilov, G. E.
\textit{Generalized Functions II};
Academic Press, New York, 1968.

\bibitem{TT0} Teofanov, N.; Tomi\'c, F.
Extended Gevrey regularity via weight matrices.
{\em Axioms} {\bf 2022}, {\em 11 (10)}, 576, 6pp.

\bibitem{MT} Meise, R.; Taylor, B.A.
Whitney's extension theorem for ultradifferentiable functions of Beurling type.
{\em Ark. Mat.} {\bf 1988}, {\em 26}, 265--287.

\bibitem{Bonet} Bonet, J.; Meise, R.; Melikhov. S.
A comparison of two different ways to define classes of ultradifferentiable functions.
{\em Bull. Belg. Math. Soc. Simon Stevin} {\bf 2007}, {\em 14}, 425--444.

\bibitem{RS} Rainer, A.; Schindl, G.
Composition in ultradifferentiable classes.
{\em Stud. Math.} {\bf 2014},  {\em 224}, 97--131.

\bibitem{PilipovicKnjiga} Carmichael, R.; Kaminski, A.; Pilipovi\'c, S. \textit{Notes on Boundary Values in Ultradistribution Spaces};
Lecture Notes Series of Seul University, 49, 1999.

\bibitem{Komatsuultra3} Komatsu, H.
Ultradistributions, III: Vector valued ultradistributions and the theory of kernels.
{\em J. Fac. Sci. Univ. Tokyo, Sect. IA Math.} {\bf 1982}, {\em 29}, 653--718.

\bibitem{TT-01} Teofanov, N.; Tomi\'c, F.
Inverse closedness and localization in extended Gevrey regularity.
{\em J. Pseudo-Differ. Oper. Appl.} {\bf 2017}, {\em 8}, 411--421.

\bibitem{Foland}  Foland,  G. B.
\textit{Harmonic analysis in phase space};
Princeton Univ. Press,  1989.

\bibitem{PrangoskiPilip} Pilipovi\'c, S.; Prangoski, B.
On the characterizations of wave front sets via short-time Fourier transform.
{\em Math. Notes} {\bf 2019}, {\em 105 (1-2)}, 153--157.

\bibitem{TT-stft} Teofanov, N.; Tomi\'c, F.
Extended Gevrey regularity via the short-time Fourier transform.
In {\em Advances in Micro-Local and Time-Frequency Analysis}; P. Boggiatto et al. Eds.; Applied Numerical and Harmonic Analysis, Birkha\"user, 2020; pp. 455--474.

\bibitem{Schindl} Schindl, G.
Exponential Laws for Classes of Denjoy-Carleman Differentiable Mappings. PhD Thesis, Universit\"at Wien, 2013,
available online at http://othes.univie.ac.at/32755/1/2014-01-26$\_$0304518.pdf.

\bibitem{Rainer}  F\" urd\" os, S.; Nenning, D.N.; Rainer, A.; Schindl, G.
Almost analytic extensions of ultradifferentiable functions with applications to microlocal analysis.
{\em J. Math. Anal. Appl.} {\bf 2020}, {\em 481 (1)}, 123451, 51 pp.

\bibitem{DebVin} Debrouwere, A.; Vindas, J.
Quasianalytic functionals and ultradistributions as boundary values of harmonic functions.
{\em Publ. Res. Inst. Math. Sci.} {\bf 2023}, {\em 59}, 657--686.



\end{thebibliography}
\end{document}